\documentclass[12pt]{amsart}
\usepackage{amsmath, amsfonts, amssymb,amsthm}
\usepackage{euscript}
\usepackage[T1]{fontenc}
\usepackage{graphicx}
\usepackage{color}
\usepackage{xcolor}

\usepackage{hyperref}

\newtheorem{thmx}{Theorem}

\newtheorem{thm}{Theorem}[section]
\newtheorem{lem}[thm]{Lemma}

\newtheorem{prop}[thm]{Proposition}
\newtheorem{cor}[thm]{Corollary}

\theoremstyle{definition}
\newtheorem{defn}[thm]{Definition}

\newtheorem{ex}[thm]{Example}
\newtheorem{rem}[thm]{Remark}

\DeclareMathOperator{\R}{\mathbb R}
\DeclareMathOperator{\Af}{\mathbb A}
\DeclareMathOperator{\C}{\mathbb C}
\DeclareMathOperator{\N}{\mathbb N}

\DeclareMathOperator{\V}{\mathcal V}
\DeclareMathOperator{\SR}{\mathcal K^0}

\DeclareMathOperator{\SO}{\mathcal O}

\DeclareMathOperator{\K}{\mathcal K}
\DeclareMathOperator{\QQ}{\mathbb Q}
\DeclareMathOperator{\PP}{\mathbb P}

\DeclareMathOperator{\car}{char}

\DeclareMathOperator{\Reg}{Reg}

\DeclareMathOperator{\Sp}{Spec}

\DeclareMathOperator{\p}{\mathfrak{p}}
\DeclareMathOperator{\q}{{\mathfrak q}}

\DeclareMathOperator{\Mor}{\rm Mor}
\DeclareMathOperator{\Max}{\rm Max}

\DeclareMathOperator{\BAB}{B\otimes_{A} B}
\DeclareMathOperator{\CAC}{C\otimes_{A} C}

\DeclareMathOperator{\cAc}{1\otimes_A c ~-~ c\otimes_A 1}

\def\NilRad {{\rm{NilRad}}}

\def\JRad {{\rm{Rad}}}

\begin{document}

\title[\tiny{Algebraic characterizations of homeomorphisms between algebraic varieties}]{Algebraic characterizations of homeomorphisms between algebraic varieties}
\author{Fran\c cois Bernard, Goulwen Fichou, Jean-Philippe Monnier and Ronan Quarez}
\thanks{The second author wishes to thank Olivier Wittenberg for fruitful discussions. The authors have received support from the Henri Lebesgue Center ANR-11-LABX-0020-01 and the project EnumGeom ANR-18-CE40-0009.}

\address{François Bernard\\
   Univ Angers, CNRS, LAREMA, SFR MATHSTIC, F-49000 Angers, France}
\email{francois.bernard@univ-angers.fr}

\address{Goulwen Fichou\\
Univ Rennes, CNRS, IRMAR - UMR 6625, F-35000 Rennes, France}
\email{goulwen.fichou@univ-rennes1.fr}

\address{Jean-Philippe Monnier\\
   Univ Angers, CNRS, LAREMA, SFR MATHSTIC, F-49000 Angers, France}
\email{jean-philippe.monnier@univ-angers.fr}

\address{Ronan Quarez\\
Univ Rennes\\
Campus de Beaulieu, 35042 Rennes Cedex, France}
\email{ronan.quarez@univ-rennes1.fr}
\date\today

\begin{abstract} We address the question of finding algebraic properties that are respectively equivalent, for a morphism between algebraic varieties over an algebraically closed field of characteristic zero, to be an homeomorphism for the Zariski topology and for a strong topology that we introduce. Our answers involve a study of seminormalization and saturation for morphisms between algebraic varieties, together with an interpretation in terms of continuous rational functions on the closed points of an algebraic variety. The continuity refers to the strong topology which is the usual Euclidean topology in the complex case, whereas it comes from the theory of real closed fields otherwise.
\end{abstract}

\maketitle

\vskip 15mm

Let $k$ be an algebraically closed field. Let $\pi:Y\to X$ be a morphism between algebraic varieties over $k$ and $\pi_k:Y(k)\to X(k)$ be its restriction to the closed points. The main purpose of this paper is to find algebraic characterizations for topological conditions on $\pi$ or $\pi_k$.
In this direction, we compare bijections, isomorphisms and homeomorphisms with respect to the Zariski topology, and to a strong topology on the closed points when char($k$)=0 (corresponding to the Euclidean topology when $k=\C$ is the field of complex numbers).

As first comparisons, recall from the Nullstellensatz that $\pi$ is a homeomorphism if and only if $\pi_k$ is a homeomorphism. Also, it is clear that if $\pi$ is an isomorphism, then it is a homeomorphism and in particular $\pi_k$ is a bijection. However, in general, having a bijection at the level of closed points does not induce a homeomorphism or an isomorphism between the varieties. To obtain results of this kind, one must add some conditions on the varieties. For example, if $\pi_k$ is a bijection and $X$,$Y$ are irreducible curves, then $\pi$ is a homeomorphism. In greater dimension, a bijection (or a birational bijection in positive characteristics) between irreducible varieties induces an isomorphism when the target variety is normal by Zariski Main Theorem. 

Assuming now $\pi$ to be a homeomorphism, there are similar results involving the notions of seminormality and weak normality instead of normality. Andreotti and Bombieri \cite{AB} proved that $\pi$ is an isomorphism if $X$ is weakly normal and $\pi$ is  finite. Vitulli \cite{V2} managed to remove the finiteness assumption on $\pi$, by requiring that $X$ does not have one-dimensional components. This dimensional condition is necessary, as illustrated by the normalization of a nodal curve with one of the preimage of the singular point removed : we obtain a homeomorphism onto a seminormal curve which is not an isomorphism. The dimensional condition of Vitulli guaranties in fact that the homeomorphism is a finite morphism.

The weak normality property appearing in the results of Andreotti, Bombieri and Vitulli is very closed to the notion of seminormality and they are both related to the notions of subintegral and weakly subintegral morphisms. A morphism $\pi:Y\to X$ is (resp. weakly) subintegral if it is integral, bijective and equiresidual (resp. residually purely inseparable) i.e. it induces isomorphisms (resp. purely inseparable extensions) between the residue fields. The seminormalization (resp. weak normalization) of $X$ in $Y$, introduced by Bombieri and Andreotti, is the maximal variety with a (resp. weakly) subintegral morphism onto $X$ which factorizes $\pi$. Note that when $\car(k)=0$, weakly subintegral means subintegral and weak normality means seminormality.
The notions of weak normality and seminormality were first introduced by Andreotti and Norguet \cite{AN} for complex analytic varieties, then by Andreotti and Bombieri \cite{AB} for schemes and by Traverso for rings \cite{T}. It appears in the study of Picard groups \cite{T} or as singularities in the minimal model program \cite{KoKo}. Seminormalization, weak normalization and (weakly) subintegral morphisms are studied in section \ref{sect-semi}. The close notion of radicial morphism, as introduced by Grothendieck \cite{Gr1}, only requires that the (non necessarily integral) morphism is injective and residually purely inseparable. We study in section \ref{sect-sat} the saturation for varieties as the geometric counterpart of radiciality. The saturation appears first in the context of Lipschitz geometry with works of Pham and Teissier \cite{PT} in complex analytic geometry and Lipman \cite{Lip} for ring extensions. For integral morphisms,  saturation and weak normalization coincide, providing a different approach to weak normality as proposed by Manaresi \cite{Mana}. However, it is not established that the saturation produces a variety, contrarily to the weak normalization and the seminormalization. A comparaison between all these notions is done in section \ref{comparaison}.

It will be crucial in our discussions to consider another topology than the Zariski topology.
For an algebraic variety $X$ over the complex numbers, the strong topology on the closed points $X(\C)$ of $X$ comes from the isomorphism $\R[\sqrt{-1}]=\C$ that gives an identification $\C\simeq \R^2$ and the property of $\R$ to be a real closed field. Indeed, the Euclidean topology on $\R^n$ has a basis of open sets given by semialgebraic subsets of $\R^n$, i.e. given by real polynomial equalities and inequalities. The theory of semialgebraic sets provides an algebraic way to discuss about topological question in real algebraic geometry, as developed in \cite{BCR}. The great advantage of this approach is that it generalizes from $\R$ to any real closed field. A real closed field is an ordered field that does not admit any ordered algebraic extension. Equivalently, adding a square root of $-1$ to a real closed field gives an algebraically closed field of characteristic zero. Real closed fields have been initially studied by Artin and Schreier \cite{AS} in the way of Artin's proof of Hilbert XVIIth Problem \cite{A}. The most basic examples away from $\R$ are the field of algebraic real numbers, which is the real closure of $\QQ$, and the field of Puiseux series with real coefficients, which is the real closure of the field $\R((T))$ of Laurent series ordered by $T$ positive and infinitely small. There are many of them, as illustrated by the fact that any algebraically closed field $k$ of characteristic zero contains (infinitely many) real closed subfields $R\subset k$ with $k=R[\sqrt{-1}]$. Fixing such a choice of $R$ leads to an identification $k\simeq R^2$ and equips $k$ with an order topology, called the $R$-topology on $k$. Note that in general $R$ is not connected, and a closed and bounded interval is not compact. Anyway, for a given algebraic variety $X$ over an algebraically closed field $k$ of characteristic zero, the choice of a real closed field $R\subset k$ such that $k=R[\sqrt{-1}]$ makes $X(k)$ a topological space for the $R$-topology. Of course, if $k=\C$ and $R=\R$ then we recover the Euclidean topology. The $R$-topology on the closed points of algebraic varieties over an algebraically closed field $k$ of characteristic zero is studied in section \ref{regulous}. The introduction of the $R$-topology allows us to characterize finite morphisms which are homeomorphisms via subintegrality and radiciality. Concerning subintegrality, it generalizes to any algebraically closed field of characteristic zero a result of the first author \cite[Thm. 3.1]{Be} in complex geometry.

\begin{thmx}\label{thmA} Let $\pi:Y\to X$ be a finite morphism between algebraic varieties over an algebraically closed field $k$ of characteristic zero. Let $R\subset k$ be a real closed subfield such that $k=R[\sqrt{-1}]$.
The following properties are equivalent :
\begin{enumerate}
\item[(i)] $\pi$ is a homeomorphism.
\item[(ii)] $\pi_k$ is a homeomorphism for the $R$-topology.
\item[(iii)] $\pi$ is subintegral.
\item[(iv)] $\pi$ is radicial.
\end{enumerate}
\end{thmx}

We focus in section \ref{sect-homeo} on the relations between the four equivalent properties appearing in Theorem \ref{thmA} when we remove the finiteness hypothesis. Note that the first two properties are topological whereas the last ones are algebraic. The equivalence between the four above properties is no longer true without the finiteness hypothesis. In particular, a homeomorphism with respect to the Zariski topology need not be a homeomorphism with respect to the $R$-topology, even for irreducible affine curves. Our interpretation is that the relevant topology to associate to subintegrality is the $R$-topology, whereas Zariski topology is rather related to the notion of radiciality. 

\begin{thmx}\label{thmB} Let $\pi:Y\to X$ be a morphism between algebraic varieties over an algebraically closed field $k$ of characteristic zero. Let $R\subset k$ be a real closed subfield such that $k=R[\sqrt{-1}]$.
Then :
\begin{enumerate}
\item[(i)] $\pi_k$ is a homeomorphism for the $R$-topology if and only if $\pi$ is subintegral.
\item[(ii)] If $\pi$ is a homeomorphism then $\pi$ is radicial.
\end{enumerate}
\end{thmx}

The key result we prove to get the first statement of Theorem \ref{thmB} is that a homeomorphism for the $R$-topology is always finite. Note moreover that, if the second part of Theorem \ref{thmB} does not refer to the $R$-topology, our proof is a consequence of the first result where the use of the $R$-topology is crucial. We prove along the way that a homeomorphism with respect to the $R$-topology is a homeomorphism, the converse being true except for curves. In section \ref{sect-carpos} we consider the situation where $k$ has positive characteristic, and then the $R$-topology does not exist anymore. We provide an alternative proof of the second statement of Theorem \ref{thmB} in this context. As consequences, we completely answer the question considered by Vitulli \cite{V2} of characterizing when a homeomorphism is an isomorphism. 

\begin{thmx}\label{thmC} Let $\pi:Y\to X$ be a morphism between algebraic varieties over an algebraically closed field $k$ of characteristic zero. Let $R\subset k$ be a real closed subfield such that $k=R[\sqrt{-1}]$.
\begin{enumerate}
\item[(i)] Assume $\pi_k$ is a homeomorphism for the $R$-topology. Then, $\pi$ is an isomorphism if and only if $X$ is seminormal in $Y$.
\item[(ii)] Assume $\pi$ is a homeomorphism. Then, $\pi$ is an isomorphism if and only if $X$ is saturated in $Y$.
\end{enumerate}
\end{thmx}

The study of seminormalization over non algebraically closed fields presents some difficulty, and the last three authors managed to define a sort of seminormalization for algebraic varieties over the field of real numbers \cite{FMQ}. This notion has to do with the central points of a real algebraic variety, that is the Euclidean closure of the set of regular points. This approach would not have been possible without the recent study of continuous rational functions in real algebraic geometry, as initiated by Kucharz \cite{Ku}, Koll\'ar and Nowak \cite{KN}, and further developed in \cite{FHMM} as regulous functions. These regulous functions happens to be related to seminormality for complex algebraic varieties too as studied by the first author \cite{Be}. It is this approach of seminormality via continuous functions with respect to the $R$-topology we develop at the end of section \ref{sect-CR}.
In particular, we provide a full study of the relation between seminormality and the $R$-topology completely parallel to the complex case in \cite{Be}. We introduce the continuous rational functions over $k$ in section \ref{sect-CR}. A remarkable fact is that the continuity of a rational function defined on an algebraic variety $X$ over $k$ does not depend on the choice of the real closed field $R$. Even more, these continuous rational functions coincide with the regular functions on the seminormalization of $X$, cf. Theorem \ref{caractK0}. As a consequence, fixing a real closed field $R\subset k$ brings all the flexibility of semialgebraic geometry over $R$ to algebraic geometry over $k$, without loosing in generality. 
We prove notably that the seminormalization determines a variety up to biregulous equivalence. As a consequence, we obtain the following result, which goes in the direction of the problems considered by Koll\'ar in \cite{Ko} (or \cite{KMOS,Ce}).

\begin{thmx}\label{thmD}
\label{thmD}
Let $X$ and $Y$ be seminormal algebraic varieties over an algebraically closed field $k$ of characteristic zero. If $X(k)$ and $Y(k)$ are biregulously equivalent, then $X$ and $Y$ are isomorphic.
\end{thmx}

\vskip 5mm

In the paper, $k$ denotes a field (sometimes algebraically closed) of any characteristic, and an algebraic variety over $k$ is a reduced and separated scheme of finite type over $k$.

\section{Subintegrality, weak normalization and seminormalization}\label{sect-semi}

After some reminders on integral extensions and normalization, we recall the notions of (weakly) subintegral extensions and the respective constructions of Traverso \cite{T}, Andreotti and Bombieri \cite{AB} of the seminormalization and the weak normalization for ring extensions and morphisms between algebraic varieties. 
\vskip 5mm

{\bf Notation and terminology.}
 Let $A$ be a ring. The Zariski spectrum $\Sp A$ of $A$ is the set of prime ideals of $A$. It is a topological space for the topology whose closed sets are generated by the sets $\V(f)=\{\p\in\Sp A\mid 
f\in\p\}$ for $f\in A$. We denote by $\Max A$ the subspace of maximal ideals of $A$. For $\p\in\Sp A$, we denote by $k(\p)$ the residue field at $\p$. 

Let $(X,\SO_X)$ be a variety over $k$.
For $x\in X$ we denote by $k(x)$ the residue field at $x$; for an affine neighborhood $U$ of $x$ then $x$ corresponds to a prime ideal $\p_x$ of $\SO_X(U)$ and we have $k(x)=k(\p_x)$. In case $X$ is affine then we denote by $k[X]$ the coordinate ring of $X$ i.e $k[X]=\SO_X(X)$.
Let $K$ be a field containing $k$. We denote by $X(K)$ the set $\Mor(\Sp K,X)$ of $K$-rational points. 
If $K=k$ then $X(k)$ is also the set of $k$-closed points of $X$, i.e the points of $X$ with residue field equal to $k$. We have thus an inclusion $$X(k)\hookrightarrow X$$ that makes $X(k)$ a topological space for the Zariski topology. We denote by $\SO_{X(k)}$ the sheaf of regular functions on $X(k)$, for $x\in X(k)$ we have $\SO_{X,x}=\SO_{X(k),x}$. 
In the case $k$ algebraically closed, for an open subset $U$ of $X$, we may identify $\Max \SO_X(U)$ with $U(k)$ by the Nullstellensatz, and similarly we identify the regular functions on
$U$ with those on $U(k)$, namely $\SO_X(U)=\SO_{X(k)}(U(k))$. If $T$ is a subset of $X$ or $X(k)$ or $\Sp A$ then we will denote by $\overline{T}^Z$ the closure of $T$ for the Zariski topology.

A ring extension $i:A\to B$ induces a map $\Sp(i):\Sp B\to \Sp A$, 
given by $\p\mapsto (\p\cap A)=i^{-1}(\p)$.
If $\pi:Y\rightarrow X$ is a morphism between
algebraic varieties over $k$, with $\SO_X\to \pi_*\SO_Y$ the associated morphism of sheaves of rings on $X$, then for any open subset $U\subset X$ the ring morphism $\SO_X(U)\to \SO_Y(\pi^{-1}(U))$ is an extension if $\pi$ is dominant. For a field extension $k\to K$, we denote by $\pi_{K}:Y(K)\to X(K)$ the induced map. Remark that $\pi_k$ is also the restriction of $\pi$ to the $k$-closed points.

\vskip 2mm

In the sequel, $\K(A)$ (resp. $\K$) will denote the total ring of fractions of $A$ (resp. the sheaf of total ring of fractions on $X$).

\subsection{Reminder on integral extensions and normalization}

A ring extension $A\to B$ is said of finite type (resp. finite) if it makes $B$ a finitely generated $A$-algebra (resp. $A$-module). The extension $A\to B$ is
birational if it induces an isomorphism between 
$\K(A)$ and $\K(B)$.

An
element $b\in B$ is integral over $A$ if $b$ is the root of a monic
polynomial with coefficients in $A$, which is equivalent for $A[b]$ to be a finite $A$-module by \cite[Prop. 5.1]{AM}. As a consequence  
$$A_B'=\{b\in B|\,b\,\, {\rm is\,\,
  integral\,\, over}\,\,A\}$$ is a ring called the integral closure of $A$ in
$B$. The extension $A\to B$ is said to be integral if $A_B'=B$. In
case $B=\K(A)$ then the ring $A_{\K(A)}'$ is
denoted by $A'$ and is simply called the integral closure of $A$.
The ring $A$ is called integrally closed (resp. in $B$) if
$A=A'$ (resp. $A=A_B'$).

We recall that a dominant morphism $Y\rightarrow X$ between 
algebraic varieties over $k$ is said of finite type (resp. finite, birational,  integral) if for any open subset $U\subset X$ the ring extension
$\SO_X(U)\rightarrow \SO_Y(\pi^{-1}(U))$ is of finite type (resp. finite, birational, integral). In this paper, a morphism between algebraic varieties is always of finite type.

Let $X$ be an algebraic variety over $k$. The normalization of $X$, denoted by $X'$, is the algebraic variety over $k$ with a finite birational morphism $\pi':X'\rightarrow X$, called the normalization morphism such that for any open subset $U\subset X$ we have $\SO_{X'} (\pi'^{-1}(U))=\SO_X(U)'$.
We say that $X$ 
is normal if $\pi'$ is an isomorphism. A point $x\in X$ is said normal if $\SO_{X,x}$ is integrally closed.

\vskip 2mm

We will use frequently that an integral extension of rings induces surjectivity at the spectrum level (See \cite[Thm. 9.3]{Ma} or \cite[Thm. 5.10, Cor. 5.8]{AM}).
\begin{prop}
\label{lying-over}
Let $A\to B$ be an integral extension of rings.
The maps $\Sp B\to \Sp A$ and $\Max B\to \Max A$ are surjective and closed.
\end{prop}

As a consequence, if $\pi$ is a finite morphism between
algebraic varieties over $k$
then Proposition \ref{lying-over} implies that $\pi$ and $\pi_{k}$ are surjective.

\subsection{Subintegral extensions, weak normalization and seminormalization}

We recall the concept of subintegral and weakly subintegral extensions introduced respectively by Traverso \cite{T}, Andreotti and Bombieri \cite{AB}.

\begin{defn}
Let $A\to B$ be an extension of rings.
\begin{enumerate}
\item For $\p\in\Sp B$, we say that $\Sp B\to\Sp A$ is equiresidual (resp. residually purely inseparable) at $\p$ if the extension $k(\p\cap A)\to k(\p)$ is an isomorphism (resp. purely inseparable).\\ 
Let $W\subset \Sp B$, we say that $\Sp B\to\Sp A$ is equiresidual  (resp. residually purely inseparable) by restriction to $W$ if for any $\p\in W$, $\Sp B\to\Sp A$ is equiresidual (resp. residually purely inseparable) at $\p$. If $W=\Sp B$ then we simply say equiresidual (resp. residually purely inseparable).\\
The extension $A\to B$ is said equiresidual (resp. residually purely inseparable) if $\Sp B\to \Sp A$ is.
\item The extension $A\to B$ and the map $\Sp B\to \Sp A$ are said (resp. weakly) subintegral if the extension is integral and $\Sp  B\to\Sp A$ is bijective and equiresidual (resp. residually purely inseparable).
\end{enumerate}
\end{defn}

Note that a field extension is equiresidual (resp. residually purely inseparable) if and only if it is an isomorphism (resp. purely inseparable). Remark that a subintegral extension is weakly subintegral and that the converse holds in characteristic zero.
We extend these definitions to the geometric setting, adding moreover a notion of hereditarily birational morphism that have been introduced in the real setting in \cite{FMQ}.

\begin{defn}
Let $\pi:Y\to X$ be a dominant morphism between
algebraic varieties over $k$. 
\begin{enumerate}
\item We say that $\pi$ is equiresidual (resp. residually purely inseparable) if for any $y\in Y$ then the field extension $k(\pi(y))\to k(y)$ is an isomorphism (resp. purely inseparable).
\item We say that $\pi$ is (resp. weakly) subintegral if $\pi$ is integral, bijective and equiresidual (resp. residually purely inseparable).
\item We say that $\pi:Y\to X$ is hereditarily birational if for any open subset $U\subset X$ and 
for any irreducible algebraic subvariety $V=\V(\p)\simeq\Sp(\SO_Y(\pi^{-1}(U))/\p)$ in $\pi^{-1}(U)$, the
morphism $$\pi_{|V}: V\to W=\V(\p\cap \SO_X(U))\simeq\Sp\big(\SO_X(U)/(\p\cap \SO_X(U))\big)$$ is 
birational. 
\end{enumerate}
\end{defn}

Geometrically speaking, a dominant morphism $\pi:Y\to X$ is equiresidual if and only if it is hereditarily birational. Indeed, for any open subset $U\subset X$ and for any irreducible algebraic subvariety $V=\V(\p)\simeq\Sp(\SO_Y(\pi^{-1}(U))/\p)$ in $\pi^{-1}(U)$, the restricted morphism $$\pi_{|V}: V\to W=\V(\p\cap \SO_X(U))\simeq\Sp \big( \SO_X(U)/(\p\cap \SO_X(U))\big)$$ is 
birational if and only if the extension $k(\p\cap \SO_X(U))=\K(W)\to k(\p)=\K(V)$ is an isomorphism.

An hereditarily birational morphism is not necessarily bijective. However, adding an integrality assumption and using Proposition \ref{lying-over}, we get the following characterization.

\begin{lem}
Let $\pi:Y\to X$ be an integral morphism between
algebraic varieties over $k$. The following properties are equivalent:
\begin{enumerate}
 \item $\pi$ is hereditarily birational and injective.
 \item $\pi$ is subintegral.
\end{enumerate}
\end{lem}

The notion of (resp. weak) subintegral extension leads to the notion of seminormalization (resp. weak normalization), in a similar way that integral extensions lead to normalization.

In order to define the notions of seminormality and weak normalization, we need to consider sequences of ring extensions.
A ring $C$ is said intermediate between the rings $A$ and $B$ if there exists a sequence of extensions $A\to C\to B$. In that case, we say that $A\to C$ and $C\to B$ are intermediate extensions of $A\to B$ and we say in addition that $A\to C$ is a subextension of $A\to B$.

Seminormal (resp. weakly normal) extensions are maximal (resp. weakly) subintegral extensions.
\begin{defn}
\label{defseminorm}
Let $A\rightarrow C\to B$ be a sequence of two extensions of rings with $A\to C$ (resp. weakly) subintegral.
We say that $C$ is seminormal (resp. weakly normal) between $A$ and $B$ 
  if for every
  intermediate ring
  $D$ between $C$ and $B$, with $C$ different from $D$, then $A\to D$ is not (resp. weakly)
  subintegral. 
  
  We say that $A$ is seminormal (resp. weakly normal) in $B$ if $A$ is
  seminormal (resp. weakly normal) between $A$ and $B$. We say that $A$ is seminormal (resp. weakly normal) if $A$ is seminormal (resp. weakly normal) between $A$ and $A'$.
\end{defn}

Recall that the characteristic exponent $e(K)$ of a field $K$ is 1 if $\car (K) = 0$
and is $p$ if $\car (K) = p > 0$.
Given an extension of rings $A\to B$, Traverso  \cite{T}, Andreotti and Bombieri \cite{AB} (see also
\cite{V}) proved respectively there exists a unique intermediate ring which is seminormal (resp. weakly normal) between $A$ and $B$. To this purpose, they introduced the rings
$$A_B^+=\{b\in A_B'|\,\,\forall\p\in\Sp
A,\,\,b_{\p}\in A_{\p}+\JRad((A_B')_{\p})\},$$
$$A_B^*=\{b\in A_B'|\,\,\forall\p\in\Sp
A,\exists n\geq 0,\,\,b_{\p}^{e(k(\p))^n}\in A_{\p}+\JRad((A_B')_{\p})\},$$
where $\JRad$ stands for the Jacobson radical, namely the intersection of all the maximal ideals. The idea to build $A_B^+$ and $A_B^*$ is, for all $\p\in\Sp A$, to glue together all the prime ideals of $A_B'$ lying over $\p$ (see \cite{Mnew}).

\begin{thm}\label{thmT} \cite{T}, \cite{AB}\\
Let $A\to B$ be an extension of rings. Then $A_B^+$ (resp. $A_B^*$) is the unique ring which is seminormal (resp. weakly normal) between $A$ and $B$. 

Moreover, for any intermediate ring $C$ between A and B, the extension $ A\to C$ is (resp. weakly) subintegral if and only if $C\subset A^+_B$ (resp. $C\subset A_B^*$).
\end{thm}

The ring $A^+_B$ (resp. $A_B^*$) is called the seminormalization (resp. weak normalization) of the ring extension $ A\to B$ or the seminormalization (resp. weak normalization) of $A$ in $B$. Note that $$A\subset A_B^+\subset A_B^*\subset A_B^\prime\subset B.$$ The ring $A^+_{A'}$ (resp. $A^*_{A'}$) is called the seminormalization (resp. weak normalization) of $A$ and is simply denoted by $A^+$ (resp. $A^*$).
Note that when $A$ and $B$ are domains, then $A$ and $A_B^+$ have in particular the same fraction field and that $\K(A)\to \K(A_B^*)$ is purely inseparable.

Note that the inclusion $A_B^+\subset A_B^*$ can be strict.
\begin{ex} Let $K$ be a field of characteristic $2$, $x$ be an indeterminate, and consider
the integral extension $A=K[x^2]\to B=K[x]$. 
It follows from a criterion of Hamman that $A_B^+=A$ (see \cite[Ex. 2.13]{V}). Since $x^2$ and $2x= 0$ are both
in $A$ but $x$ is not in $A$, it follows from \cite[Prop. 3.10]{V} that $A$ is not weakly normal in $B$.
\end{ex}

\subsection{Seminormalization and weak normalization of a morphism between algebraic varieties}
  
Andreotti and Bombieri \cite{AB} have introduced and built the seminormalization and the weak normalization of a scheme in another one.
In this section, we provide a different and elementary construction of the seminormalization and the weak normalization of an affine algebraic variety in another one. 
The seminormalization and the weak normalization answer the respective following questions. Let $Y\to X$ be a dominant morphism between algebraic varieties over $k$. Does there exist a biggest algebraic variety $Z$ such that $Y\to X$ factorizes through $Z$ and $Z\to X$ is (resp. weakly) subintegral~?

\vskip 2mm

We recall first the notion of normalization of a variety in another one. 
Let $\pi:Y\to X$ be a dominant morphism between algebraic varieties over $k$. The integral closure $(\SO_X)_{\pi_*(\SO_Y)}'$ of $\SO_X$ in $\pi_*(\SO_Y)$ is a coherent sheaf \cite[Lem. 52.15]{STPmorph} and by 
\cite[II Prop. 1.3.1]{Gr2} it is the structural sheaf of a variety over $k$. 
\begin{defn} \label{defnorm}
Let $\pi:Y\to X$ be a dominant morphism of finite type between algebraic varieties over $k$. 
The variety with structural sheaf equal to the integral closure of  $\SO_X$ in $\pi_*(\SO_Y)$ is called the normalization of
$X$ in $Y$ and is denoted by $X_Y'$.
\end{defn}

Be aware that the normalization of a variety in another one is not necessarily a normal variety, nor it admits a birational morphism onto the original variety.

For a dominant morphism $Y\to X$ between algebraic varieties over $k$, we say that an algebraic variety $Z$ over $k$ is intermediate between $X$ and $Y$ if $Y\to X$ factorizes through $Z$. For affine varieties, it is equivalent to say that $k[Z]$ is an intermediate ring between $k[X]$ and $k[Y]$.
The normalization of a variety in another one satisfies the following property :

\begin{prop} \label{PUnormalization}
Let $Y\to X$ be a dominant morphism between algebraic varieties over $k$.  Let $Z$ be an intermediate variety between $X$ and $Y$.
Then $Z\to X$ is finite if and only if it factorizes $X_Y'\to X$.
\end{prop}

We describe now an elementary construction of the seminormalization and the weak normalization of an affine algebraic variety in another one.
Let $Y\to X$ be a dominant morphism between  affine algebraic varieties over $k$. We want to check that the rings $A_1=k[X]^{+}_{k[Y]}$ and $A_2=k[X]^*_{k[Y]}$ are coordinate rings of algebraic varieties. We know that the morphism $X_Y'\to X$ is finite by Proposition \ref{PUnormalization}, so we can apply Lemma \ref{lemintermed} below to the extensions
$$k[X]\subset A_1\subset A_2 \subset k[X'_Y]=k[X]'_{k[Y]} \subset k[Y]$$
to conclude.

\begin{lem}
  \label{lemintermed}
Let $\pi:Y\to X$ be a finite morphism between  affine
algebraic varieties over $k$. Let $A$ be a ring such that
$k[X]\subset A\subset k[Y]$. Then $A$ is the coordinate ring of a unique 
affine algebraic variety over $k$ and $\pi$ factorizes through this
variety.
\end{lem}

\begin{proof}
Since $k[Y]$ is a finite module over the Noetherian ring $k[X]$ then it is a Noetherian $k[X]$-module.
Thus the ring $A$ is a finite $k[X]$-module as a submodule of a Noetherian $k[X]$-module. It follows that $A$ is a finitely generated algebra over $k$ and the proof is done.
\end{proof}

For general constructions of the seminormalization and the weak normalization, one needs to check that the  seminormalization and the weak normalization of the local charts of an affine covering glue together to give a global variety. This is done by Andreotti and Bombieri \cite{AB} using Grothendieck criterion \cite[II Prop. 1.3.1]{Gr2} concerning the quasi-coherence of sheaves. It leads to the following definitions.

\begin{defn}
Let $\pi:Y\to X$ be a dominant morphism between algebraic varieties over $k$. The seminormalization  (resp. weak normalization) of $X$ in $Y$ is the algebraic variety $X^+_Y$ 
(resp. $X_Y^*$) over $k$ with structural sheaf equal to the seminormalization (resp. weak normalization) of  $\SO_X$ in $\pi_*(\SO_Y)$.

We call $X^{+}$ (resp. $X^*$) the seminormalization (resp. weak normalization) of $X$ in its normalization $Y=X'$. We say that $X$ is seminormal in $Y$ (resp. seminormal) if $X=X^+_Y$ (resp. $X=X^+$). We say that $X$ is weakly normal in $Y$ (resp. weakly normal) if $X=X^*_Y$ (resp. $X=X^*$). 

\end{defn}

\begin{rem} Note that the seminormalization of $X$ in $Y$ is birational to $X$, even if $Y\to X$ is not birational. It is not the case for the normalization of $X$ in $Y$ and also for the weak normalization of $X$ in $Y$. We have in general
$$Y\to X'_Y\to X^*_Y\to X^+_Y\to X.$$
\end{rem}

The seminormalization and the weak normalization of a variety in another one satisfies the following universal properties~:
\begin{prop}\label{propCSEPvariety}
Let $Y\to Z\to X$ be a sequence of dominant morphisms between algebraic varieties over $k$. Then $Z\to X$ is (resp. weakly) subintegral if and only if $X^+_Y\to X$ (resp. $X_Y^*$)
factorizes though $Z$.
\end{prop}

\begin{proof}
It is a reformulation of the second part of Theorem \ref{thmT}.
\end{proof}


\section{saturation}\label{sect-sat}

In the classical study of the seminormalization, some basic properties such as the local nature happen to be not so straightforward to prove. A nice algebraic approach has been proposed by Manaresi \cite{Mana}, in the spirit of the relative Lipschitz saturation \cite{Lip}, via the saturation of a ring $A$ in another ring $B$ which is integral over $A$. The saturation coincides with the weak normalization when the ring extension is finite. We aim to study the properties of this saturation for more general extensions, and establish its universal properties.

\subsection{Universal property of the saturation}

We define the saturation of a ring extension analogously to \cite{Mana}, but for non-necessarily integral extensions.

\begin{defn}
Let $A\to B$ be an extension of rings. The saturation of $A$ in $B$, denoted by $\widehat{A}_B$, is defined by
$$\widehat{A}_B=\{b\in B\mid b\otimes_A1-1\otimes_A b\in \NilRad (\BAB) \}$$
where the nil radical $\NilRad$ denotes the ideal of nilpotent elements.

We say that $A$ is saturated in $B$ if $\widehat{A}_B=A$.
The saturation of $A$ is its saturation in $A'$ and it is simply denoted by $\widehat{A}$. We say that $A$ is saturated if $\widehat{A}=A$. 
\end{defn}

Recall that the nilradical is the intersection of all prime ideals.
In order to study the saturation, we need to understand better the relation between prime ideals in $A$ and $B$ and prime ideals in $\BAB$. For a ring extension $A\to B$, we introduce the notation 
$\varphi_1$, $\varphi_2$ for the ring morphisms  $\varphi_i:B\to \BAB$ defined by 
\begin{equation}\label{eq-phi}
\varphi_1(b)= b\otimes_A 1 \textrm{~~~~~and~~~~~}\varphi_2(b)= 1\otimes_A b.
\end{equation}

The data of a prime ideal $\omega$ in $\BAB$, or more precisely the data of a morphism $g:\BAB\to k(\omega)$ with kernel $\omega$, is equivalent to the data of a $4$-tuple of prime ideals $$(\p_1,\p_2,\q,\p)\in\Sp B\times\Sp B\times\Sp A\times\Sp (k(\p_1)\otimes_{k(\q)}k(\p_2))$$
such that $$\p_1=\ker (g\circ \varphi_1),~~\p_2=\ker (g\circ \varphi_2),~~\q=\p_1\cap A=\p_2\cap A,~~k(\omega)=k(\p)$$ and such that the 
composition $$\BAB\to k(\p_1)\otimes_{k(\q)}k(\p_2)\to k(\p)$$ coincides with $g$.

The saturation is a ring, compatible with inclusion.

\begin{lem}\label{lem-elem} Let $A\to B$ be an extension of rings.
\begin{enumerate}
\item $\widehat{A}_B$ is a subring of $B$ containing $A$.
\item If $A\subset C \subset B$, then $\widehat{A}_B \subset \widehat{C}_B$.
\end{enumerate} 
\end{lem}

\begin{proof}
\begin{enumerate}
\item The set $\widehat{A}_B$ is an $A$-module as the kernel of the $A$-module morphism $$B\xrightarrow{\varphi_1-\varphi_2} \dfrac{\BAB}{\NilRad (\BAB)}.$$
The stability under product comes from the identity
$$b_1b_2\otimes_A 1-1\otimes_A b_1b_2=(b_1\otimes_A 1)(b_2\otimes_A 1-1\otimes_A b_2)+(1\otimes_A b_2)(b_1\otimes_A 1-1\otimes_A b_1)$$
and the fact that $\NilRad (\BAB)$ is an ideal.
\item The image of a nilpotent element by the ring morphism $\BAB \to B\otimes_C B$ remains nilpotent.
\end{enumerate}
\end{proof}

In order to give a universal property of the saturation, we recall the notion of radicial extension introduced by Grothendieck \cite[I, def. 3.7.2]{Gr1}. We also introduce a notion of radicial sequence of extensions similarly to \cite{Mnew}, due to the lack of integrality of the ring extensions.

\begin{defn}
\begin{enumerate}
\item An extension of rings $A\to B$ is said radicial if $\Sp  B\to\Sp A$ is injective and residually purely inseparable.
\item A sequence of extensions $A\to C\to B$ of rings is said radicial if the restriction of $\Sp C\to\Sp A$ to the image of $\Sp B\to\Sp C$ is injective and residually purely inseparable.
\end{enumerate}
\end{defn}

\begin{rem} An extension (resp. a sequence of extensions) of fields $K\to K'$ (resp. $K\to K'\to K''$) is radicial if and only if $K\to K'$ is purely inseparable.
\end{rem}

\vskip 2mm

The saturation furnishes radicial sequences of extensions.

\begin{prop}\label{prop-sat} Let $A\to B$ be a ring extension. For any $C\subset \widehat A_B$, the sequence $A\to C\to B$ is radicial.
\end{prop}
 
Before entering into the proof, we state an elementary result about field extensions. Remark that it gives a proof of (2) implies (3) of Theorem \ref{PU1saturation} in the special case of a sequence of field extensions.

\begin{lem}\label{lem-field} Let $K\to K'\to K''$ be a non radicial sequence of field extensions i.e such that $K\to K'$ is not purely inseparable. Then, there are two $K$-morphisms $K''\to L$ into a (algebraically closed) field $L$ whose compositions with $K'\to K''$ are distinct.
\end{lem} 
 
\begin{proof}
Since $K\to K'$ is not purely inseparable then it follows from \cite[Prop. I.3.7.1]{Gr1} that there are two distinct $K$-morphisms $\psi_1, \psi_2 : K'\to L'$ into a field $L'$. The point is to extend them to $K''$.

For $i\in\{1,2\}$, one can embed the field extensions $K'\to K''$ and $\psi_i: K' \to L'$ into a common extension $K'\to L'_i$ by amalgamation \cite[Chap 5, §4, Prop. 2]{Bour}. Denote by $\psi_i' :K''\to L'_i$ the induced extension.
By amalgamation of $L'_1$ and $L_2'$ over $K''$, one can assume that $\psi_1'$ and $\psi_2'$ take values in a common field $L$. 

The morphisms $\psi_1',\psi_2': K'' \to L$ fulfil the requirements since the restriction of $\psi_i'$ to $K'$ coincides with $\psi_i$.  
\end{proof}

\begin{rem} It is classical that one can choose $L=L'$ in the proof of Lemma \ref{lem-field} if the extension $K'\to K''$ is moreover algebraic, and this is used in \cite{Lip} to prove that the Lipschitz saturation is stable under contraction : in the setting of Proposition \ref{prop-sat}, if $C\to B$ is integral, then the Lipschitz saturation of $A$ in $C$ is equal to the intersection of $C$ with the Lipschitz saturation of $A$ in $B$. In our context the extensions are not assumed to be integral, and this contraction property does not hold, as illustrated by Example \ref{exVitdetail}.
\end{rem}

\begin{proof}[Proof of Proposition \ref{prop-sat}]
Let $\p_1$, $\p_2$ be two prime ideals of $B$ lying over the same ideal $\q$ of $A$. A first step is to prove that $\p_1$ and $\p_2$ lye over the same ideal of $C$.

Let $\p$ be a prime ideal of $k(\p_1)\otimes_{k(\q)}k(\p_2)$, and $\omega=(\p_1,\p_2,\q,\p)\in\Sp (\BAB)$ be the corresponding element, coming with a morphism $g:\BAB\to k(\omega)$ with $\ker g=\omega$. For $c\in\p_1\cap C$, we have $g\circ \varphi_1(c)=g(c\otimes_A 1)=0$ by construction of $g$. The element $\cAc$ is nilpotent in $\BAB$ by assumption, so that
 $$0=g(\cAc)=g\circ \varphi_1(c)-g\circ \varphi_2(c).$$
As a consequence $g\circ \varphi_2(c)=0$ and thus $c\in\p_2\cap C$. By symmetry we obtain $$\p_1\cap C= \p_2\cap C.$$
 
The second step is to prove that the extension $\phi:k(\q)\to k(\p_1\cap C)$ is purely inseparable, where $\q=\p_1\cap A$. Assume by contradiction that $\phi$ is not purely inseparable, and consider the composition
$$k(\q)\xrightarrow{\phi} k(\p_1\cap C)\rightarrow k(\p_1).$$
By Lemma \ref{lem-field}, there are two distinct $k(\q)$-morphisms $\psi_1,\psi_2:k(\p_1)\to L$ into some field $L$, which remain distinct by restriction to $k(\p_1\cap C)$. Thus there exists $c\in C$ such that $\psi_1\circ\pi (c)\not=\psi_2\circ\pi (c)$, where $\pi:B\to k(\p_1)$ denote the natural morphism.

The morphisms $\psi_1$ and $\psi_2$ induce a morphism $\psi:k(\p_1)\otimes_{k(\q)}k(\p_1)\to L$ given by 
$$\psi(\pi(b_1)\otimes_{k(\q)} \pi(b_2))=\psi_1\circ \pi(b_1)\cdot\psi_2\circ \pi(b_2)$$ 
for $b_1,b_2\in B$.
The kernel $\p$ of $\psi$ gives rise to a prime ideal $\omega=(\p_1,\p_2,\q,\p)$ of $\BAB$ coming with a morphism
$$g:\BAB\to k(\omega)\to L.$$ 
By our choice of $c$, the element
$$\psi(\pi(c)\otimes 1-1\otimes \pi(c))=\psi_1\circ\pi (c)-\psi_2\circ\pi (c)$$
is not zero, so that $\cAc$ does not belong to  $\ker g=\omega$, contradicting the inclusion $C\subset \widehat{A}_B$.
\end{proof}

Actually the converse of the preceding result holds true, and it gives rise to universal properties of the saturation, in terms of radicial sequences of extensions. This result is, up to our knowledge, not present in the literature.

\begin{thm}
\label{PU1saturation}
Let $A\xrightarrow{i} C\xrightarrow{j} B$ be a sequence of extensions of rings.
The following properties are equivalent:
\begin{enumerate}
\item For any field $K$, the map 
$$\Sp(j)\circ (\Mor(\Sp K, \Sp  B))\to \Mor(\Sp K, \Sp A)$$ 
$$(\Sp(j)\circ \alpha)\mapsto \Sp(i)\circ (\Sp(j)\circ \alpha)$$ is injective.
\item For any field $K$, if $\psi_1:B\to K$ and $\psi_2:B\to K$ are two field morphisms distinct by composition with $j$, then they are distinct by composition with $j\circ i$.
\item The sequence $A\to C\to B$ is radicial.
\item $j(C)\subset \widehat{A}_B$.
\item The kernel of the morphism $\CAC\to C$ defined by $c_1\otimes_A c_2\mapsto c_1c_2$ is included in the nilradical of $\BAB$. 
\end{enumerate}
\end{thm}

\begin{proof}
The equivalence between (1) and (2) is straightforward. 
Since $\ker (\CAC\to C)$ is generated by the elements of the form $c\otimes_A1-1\otimes_A c$ for $c\in C$, then (4) $\Leftrightarrow$ (5). Note that $(4)$ implies $(3)$ by Proposition \ref{prop-sat}.

\vskip 2mm

Let us prove that (3) implies (2) by contraposition. Let $\psi_1:B\to K$ and $\psi_2:B\to K$ be two morphisms in a field K such that $\psi_1\circ j\not=\psi_2\circ j$ and
$\psi_1\circ j\circ i=\psi_2\circ j \circ i$. Let $\p_1$, $\p_2$ and $\q$ denote respectively the kernels of $\psi_1$, $\psi_2$ and $\psi_1\circ j\circ i:A\to K$. For $i=1,2$ we get the following commutative diagram:
$$\begin{array}{ccccccc}
	A&\xrightarrow{i}  & C &\xrightarrow{j} & B & \xrightarrow{\psi_i} & K\\
	\downarrow&&\downarrow&& \downarrow & \nearrow & \\
	k(\q)&\rightarrow & k(\p_i\cap C)  & \rightarrow& k(\p_i) &  &\\
\end{array}$$
If $\p_1\cap C$ is not equal to $\p_2\cap C$, then $\Sp C\to \Sp A$ is not injective on the image of $\Sp B$.\\ 
If  $\p_1\cap C$ is equal to $\p_2\cap C$, then $\psi_1$ and $\psi_2$ induce two different $k(\q)$-morphisms $\psi'_\iota: k(\p_1\cap C) \to K$ since $\psi_1\circ j\not=\psi_2\circ j$. From \cite[Prop. I.3.7.1]{Gr1}, the extension $k(\q)\to k(\p_1\cap C)$ cannot be purely inseparable. 

In both cases, the extension $A\to C\to B$ is not radicial.




\vskip 2mm
Finally we prove that (2) implies (4) by contraposition. By assumption there are $c\in C$ and $\omega \in \Sp \BAB$ such that $\cAc \notin \omega$. The ideal $\omega$ comes with a morphism $g:\BAB \to K$ with $\ker g=\omega$. Consider the composition of $g$ with the morphisms $\varphi_1$ and $\varphi_2$ defined in \eqref{eq-phi}. By construction $g\circ \varphi_1:B\to K$ coincides with $g\circ \varphi_2:B\to K$ when composed with $j\circ i$, but not when composed with $j$ because $g\circ \varphi_1 (j(c))\neq g\circ \varphi_2(j(c))$. It contradicts $(2)$.
\end{proof}

If we focus on the particular case of radicial extensions rather that sequences, we recover \cite[Prop. I.3.7.1]{Gr1}, with an additional condition using the saturation.
\begin{prop}
\label{radicial}
Let $i:A\to B$ be an extension of rings and $\Sp(i):\Sp B\to\Sp A$ be the associated map.
The following properties are equivalent:
\begin{enumerate}
\item For any field $K$, the map $$\Mor(\Sp K, \Sp  B)\to \Mor(\Sp K, \Sp A)$$
$$\alpha\mapsto \Sp(i)\circ\alpha$$ is injective.
\item If $\psi_1:B\to K$ and $\psi_2:B\to K$ are two distinct morphisms then the compositions $\psi_1\circ i$ and $\psi_2\circ i$ are different.
\item $i:A\to B$ is radicial.
\item $B=\widehat{A}_B$.
\item The kernel of the morphism $\BAB\to B$ defined by $b_1\otimes_A b_2\mapsto b_1b_2$ is included in the nilradical of $\BAB$.
\end{enumerate}
\end{prop}

\begin{proof}
Direct consequence of Theorem \ref{PU1saturation}, using the fact that an extension $A\to B$ is radicial if and only if the sequence of extensions $A\to B\to B$ is so.
\end{proof}

\subsection{Saturation for varieties}

We begin with the definition of radiciality and saturation for morphisms \cite[Chap. I,3.7.2]{Gr1}. Then, we extend these definitions to sequences of morphisms.

\begin{defn}
\begin{enumerate}
\item Let $\pi:Y\to X$ be a dominant morphism between algebraic varieties over $k$.
We say that $\SO_X\to \pi_*\SO_Y$ is radicial if for any open subset $U\subset X$ the extension $\SO_X (U)\to\SO_Y (\pi^{-1}(U))$ is radicial. In this situation, we say that $\pi$ is radicial.
\item Let $Y\stackrel{\phi}{\to} Z\stackrel{\psi}{\to} X$ be a sequence of dominant morphisms between algebraic varieties over $k$. We say that $\SO_X\to \psi_*\SO_Z\to (\psi\circ\phi)_*\SO_Y$ is radicial if for any open subset $U\subset X$ the sequence of extensions $\SO_X (U)\to \SO_Z (\psi^{-1}(U))\to\SO_Y ((\psi\circ\phi)^{-1} (U))$ is radicial.
In this situation, we say that the sequence of morphisms $Y\to Z\to X$ is radicial.
\item We say that $X$ is saturated in $Y$ if $\SO_X$ is saturated in $\pi_*\SO_Y$, 
i.e for any open subset $U\subset X$ then $\SO_X(U)$ is saturated in $\SO_Y(\pi^{-1}(U))$.
\end{enumerate}
\end{defn}

\begin{rem}
\label{defsatvar}
\begin{enumerate}
\item Let $\pi:Y\to X$ be a dominant morphism between algebraic varieties over $k$. Then, $\pi$ is radicial if and only if $\pi$ is injective and residually purely inseparable.
\item Let $Y\stackrel{\phi}{\to} Z\stackrel{\psi}{\to} X$ be a sequence of dominant morphisms between algebraic varieties over $k$. Then, $Y\to Z\to X$ is radicial if and only $\psi$ is injective and residually purely inseparable by restriction to the image of $\phi$.
\end{enumerate}
\end{rem}

In order to translate the universal property of the saturation in terms of varieties, we recall the notion of universal injectivity from \cite[Chap. I, 3.4.3]{Gr1}.

\begin{defn} 
\begin{enumerate}
\item A morphism $\pi:Y\to X$ between algebraic varieties over $k$ is said universally injective if for any field extension $k\to K$, the map $\pi_K:Y(K)\to X(K)$ is injective.
\item A sequence of morphisms $Y\to Z\to X$ between algebraic varieties over $k$ is said universally injective if for any field extension $k\to K$, the map $Z(K)\to X(K)$ is injective by restriction to the image of $Y(K)\to Z(K)$.
\end{enumerate}
\end{defn}

Grothendieck \cite[Prop. 3.7.1]{Gr1} proved that the notions of radicial and universally injective morphisms coincide.
The universal property given in Theorem \ref{PU1saturation} implies that it is also the case if we consider sequences of morphisms rather than morphisms.

\begin{prop}\label{PU1saturationvar}
Let $Y\stackrel{\phi}{\to} Z\stackrel{\psi}{\to} X$ be a sequence of dominant morphisms between algebraic varieties over $k$. The following properties are equivalent:
\begin{enumerate}
\item $Y\to Z\to X$ is universally injective.
\item $Y\to Z\to X$ is radicial.
\item $\psi_*\SO_Z\subset (\widehat{\SO_X})_{(\psi\circ\phi)_*\SO_Y}$ i.e for any open subset $U\subset X$ we have $\SO_Z (\psi^{-1}(U))\subset\widehat{\SO_X (U)}_{\SO_Y ((\psi\circ\phi)^{-1} (U))}$.
\end{enumerate}
\end{prop}

\begin{rem} Let $\pi:Y\to X$ be a dominant morphism between affine algebraic varieties over $k$. Contrarily to the seminormalization case it is not clear whether $\widehat{k[X]}_{k[Y]}$ is a finitely  generated algebra over $k$ and thus lead to the existence of a variety. 
\end{rem}

As a consequence of the previous remark, the statement $\pi$ is subintegral if and only if $X^+_Y=Y$ has no equivalent when $\pi$ is radicial. Nevertheless we get:
\begin{cor}
\label{PU2saturationvar}
Let $\pi:Y\to X$ be a dominant morphism between algebraic varieties over $k$. Then $\pi$ is radicial if and only if 
$(\widehat{\SO_X})_{\pi_*\SO_Y}=\pi_*\SO_Y$.
\end{cor}

\section{Comparison between saturation, seminormalization and weak normalization} \label{comparaison}

In general, the seminormalization and the weak normalization are only included in the saturation.

\begin{lem} \label{satetsemi}
Let $A\to B$ be an extension of rings. Then $$A^+_B\subset A^*_B\subset \widehat{A}_B.$$
\end{lem}

\begin{proof}
We already know that $A^+_B\subset A^*_B$.
Since $A\to A^*_B$ is weakly subintegral then $A\to A^*_B\to B$ is radicial. The inclusion $A^*_B\subset \widehat{A}_B$ follows by Theorem \ref{PU1saturation}.
\end{proof}

If $A\to B$ is a purely inseparable extension of fields which is not an isomorphism then we get $A^+_B\subset A^*_B= \widehat{A}_B$ and the inclusion is strict. So in the sequel we focus in the comparison between weak normalization and saturation.

Note that there is no special relationship between the saturation and the relative normalization. However saturation and weak normalization coincide when we restrict to integral extensions.
\begin{prop} \label{satetsemi2}
Let $A\to B$ be an integral extension of rings. Then $$A^*_B= \widehat{A}_B.$$
\end{prop}

\begin{proof}
The direct inclusion comes from Lemma \ref{satetsemi}.

 The sequence $A\to  \widehat{A}_B\to B$ is radicial by Theorem \ref{PU1saturation}. Since $ \widehat{A}_B\to B$ is integral then $\Sp B\to \Sp \widehat{A}_B$ is surjective by Proposition \ref{lying-over}. It follows that $A\to  \widehat{A}_B$ is radicial and integral and thus is weakly subintegral. This forces $ \widehat{A}_B$ to be equal to the weak normalization $A^*_B$ of $A$ in $B$ by Theorem \ref{thmT}.
\end{proof}

Finally we state the relations between saturation, weak normalization and seminormalization for varieties induced by Lemma \ref{satetsemi} and Proposition \ref{satetsemi2}.

\begin{prop} \label{sat=semivar}
Let $\pi:Y\to X$ be a dominant morphism between varieties over $k$. 
\begin{enumerate}
\item If $\pi$ is subintegral then $\pi$ is weakly subintegral. 
\item If $\pi$ is weakly subintegral then $\pi$ is radicial.
\item If $X$ is saturated in $Y$, then $X$ is weakly normal and seminormal in $Y$. 
\item If $\pi$ is moreover integral, then $\pi$ is weakly subintegral if and only if $\pi$ is radicial, and $X$ is saturated in $Y$ if and only if $X$ is weakly normal in $Y$.
\end{enumerate}
\end{prop}

In the following proposition and examples $k$ is an algebraically closed field and $\car (k)=0$.

We compare the notions of seminormality and relative seminormality.
\begin{prop} \label{semimprel}
Let $\pi:Y\to X$ be a dominant morphism between varieties over $k$. If $X$ is seminormal then $X$ is seminormal in $Y$.
\end{prop}

\begin{proof}
Suppose $X$ is not seminormal in $Y$. From Proposition \ref{propCSEPvariety} the morphism $\varphi:X_Y^+\to X$ is subintegral, factorizes $\pi$ and is not an isomorphism. Since $\varphi$ is birational and finite then it follows that the normalization map $X'\to X$ factorizes throught $\varphi$. By subintegrality of $\varphi$ and seminormality of $X$ then we get a contradiction.
\end{proof}

The converse of Proposition \ref{semimprel} is false, take $Y=X$ with $X$ not seminormal for example.
\medskip

From \cite[Rem. 1.4]{C}, we know that the notions of relative weak normalization and relative saturation differ in any characteristic when we do not consider integral extensions of rings and integral morphisms of varieties.
We end this section by providing explicit examples to illustrate that the notions of relative saturation and seminormalization do not coincide for varieties over an algebraically closed field of characteristic null, in any dimension. The examples are built on Example \ref{exVitdetail}, constructed from a nodal curve, for which we offer two arguments : a simple geometric one, and a direct computational one in order to construct the generalization in any dimension in Example \ref{exVitdetail2}.

\begin{ex} \label{exVitdetail}
\begin{enumerate}
\item Let $X$ be the nodal plane curve with coordinate ring $A=k[X]=k[x,y]/(y^2-x^2(x+1))$. Its normalization $X'$ has coordinate ring $A'=k[X']=k[x,z]/(z^2-(x+1))=A[y/x]$ and the inclusion $A\to A'$ is given by $(x,y)\mapsto (x,xz)$. Let $Y$ be defined by removing one of the two points $p=(0,1)$ and $q=(0,-1)$ of $X'(k)$ lying above the singular point of $X(k)$, say $p$. The coordinate ring of $Y$ is $$B=k[Y]=k[x,z,s]/(z^2-(x+1),s(z-1)-1)=A'[1/(z-1)]=A'[s]=A[y/x,s],$$
and we have a sequence of inclusions $A\to A'\to B$.

Then $A^+_B=A$ whereas $\widehat{A}_B=B$. To see the first point, since the variety $X$ is seminormal (see \cite{GT}) then $X$ is seminormal in $Y$ by Proposition \ref{semimprel}. For the second point, note that $A\to B$ is radicial because, for irreducible curves, the prime ideals correspond to the generic point and the closed points, and here $Y\to X$ is birational with $Y(k)\to X(k)$ bijective. As a consequence $\widehat{A}_B=B$ by Proposition \ref{radicial}.

\item We revisit the nodal curve example proving the equality $\widehat{A}_B=B$ using the very definition of the saturation. Keeping previous notation, set $\alpha=(z\otimes_A 1)-(1\otimes_A z)$ and $\beta=(s\otimes_A 1)-(1\otimes_A s)$. It suffices to prove that $\alpha$ and $\beta$ are nilpotent elements of $\BAB$. Indeed $\widehat{A}_B$ is a ring containing $x,z$ and $s$ so that $\widehat{A}_B=B$ in that case.

Note that
$$\begin{array}{cccc}
x \alpha  & = & x \big((z+1)\otimes_A 1-1\otimes_A(z+1)\big) & \\
&= &  \big(x(z+1)\otimes_A 1\big) - \big(1\otimes_A x(z+1)\big) &\\
&= &  \big((y+x)\otimes_A 1\big) - \big(1\otimes_A (y+x)\big) & \textrm{~~since~~} y+x=x(z+1) \textrm{~~in~~} B\\
&= &  0 &\\
\end{array}$$
hence
$$\begin{array}{cccc}
\alpha^2  & = & \alpha \big((z+1)\otimes_A 1-1\otimes_A(z+1)\big) & \\
  & = & \alpha (xs\otimes_A 1-1\otimes_Axs) & \textrm{~~since~~} xs=z+1 \textrm{~~in~~} B\\
&= &  x\alpha (s\otimes_A 1 - 1\otimes_A s) &\\
&= &  0. &\\
\end{array}$$
Actually we even have $\alpha=0$ in $\BAB$, since a straightforward computation shows that $\alpha=\frac1{4}\alpha^3$. Finally, using the equality $\alpha= (z-1)\otimes_A 1-1\otimes_A(z-1)$ and the relation $s(z-1)=1$, we observe that 
$$0=(s\otimes_A 1)\alpha (1\otimes_A s)  =-\beta.$$
\end{enumerate}
\end{ex}

\begin{ex}\label{exVitdetail2} Consider the curves $X$ and $Y$ as in Example \ref{exVitdetail}. For $n\geq 1$, the variety $X\times\Af_k^n$ is seminormal in $Y\times\Af_k^n$, whereas the saturation of $X\times\Af_k^n$ in $Y\times\Af_k^n$ is $Y\times\Af_k^n$. 

To see this, note that $X$ and $\Af_k^n$ are seminormal, so $X\times \Af_k^n$ is also seminormal \cite[Cor. 5.9]{GT} and thus $X\times \Af_k^n$ is seminormal in $Y\times\Af_k^n$ by Proposition \ref{semimprel}.

For the saturation, if the radiciality of $$k[X\times\Af_k^n]= A[t_1,\ldots,t_n] \to B[t_1,\ldots,t_n]=k[Y\times\Af_k^n]$$ is not so straightforward since we no longer work with curves as in Example \ref{exVitdetail} (1), the computations done in Example \ref{exVitdetail} (2) still prove that $\widehat{A[t_1,\cdots,t_n]}_{B[t_1,\cdots,t_n]}$ contains $x,z$ and $s$, and so is equal to $B[t_1,\ldots,t_n]$.
\end{ex}

\section{Strong topology on the rational closed points and regulous functions}\label{regulous}

Continuous rational functions and regulous functions have been originally studied in real algebraic geometry \cite{Ku,KN,FHMM}, where the continuity is regarded with respect to the Euclidean topology, which can be studied algebraically via semialgebraic open sets \cite{BCR}. For an algebraic variety $X$ over $\C$, we can consider the Euclidean (or strong) topology of the complex points $X(\C)$ seen as a topological variety (see \cite{Sha2}). For example, if $X$ is affine then we  have $X\subset \mathbb A_{\C}^n$ for some $n\in\N$ and the strong topology is induced by the natural inclusion $X(\C)\subset \R^{2n}$. With this point of view Bernard has developed in \cite{Be} the theory of regulous functions for algebraic varieties over $\C$.

Subintegral extensions and regulous functions are strongly related in real algebraic geometry as developed in \cite{FMQ2}.
Working with the field of complex numbers, we know from the work of Bernard that the same holds true in the geometric case. The purpose of this section is first to generalize the work of Bernard to varieties over any algebraically closed field of characteristic zero, and then to relate it to the theory of relative seminormalization.

\vskip 2mm

In this section $k$ is an algebraically closed field and $\car (k)=0$.

\subsection{Generalizing the strong topology of $\C$}\label{sect-R}

Since there is a priori no natural strong topology on the $k$-rational points of a variety over $k$, we use the theory of real closed fields to define such a topology as in \cite{K,HK} (see also \cite{BW} for a recent cohomological use of this approach).

\vskip 2mm

From Artin Schreier theory \cite{AS}, we know the existence of (many) real closed subfields of $k$ with algebraic closure equal to $k$. Let $R\subset k$ denote one of these real closed fields. Then $k=R[\sqrt{-1}]$ and $R$ comes with a unique ordering. The ordering on $R$ gives rise to an order topology on the affine spaces $R^n$, in a similar way than the Euclidean topology on $\R^n$, even if the topological space $R$ is not connected (except in the case $R=\R$) or the closed interval $[0,1]$ is in general not compact. 

\vskip 2mm
We use this choice of $R$ to define a topology on the closed points of an algebraic variety over $k$. First, for an algebraic variety $X$ over $R$, choose an affine covering of $X$ by Zariski open subsets $U_i$, and endow each affine sets $U_i(R)$ with the order topology. These open sets glue together to define the order topology on $X(R)$, and this topology does not depend on the choice of the covering. This topological space can be endowed additionally with the structure of a semialgebraic space by considering the sheaf of continuous semialgebraic functions \cite{DK2,DK}, or even of a real algebraic variety with the sheaf of regular functions on the $R$-points \cite{BCR,Hui}.

Consider now the case of a quasi-projective algebraic variety $X$ over $k$. 
By Weil restriction \cite{W,GrFGA}, we associate to $X$ an algebraic variety $X_R$ over $R$ whose $R$-points are in bijection with the $k$-points of $X$. 
We endow $X(k)$ with the topology induced by the order topology on $X_R(R)$, and we call it the $R$-topology on  $X(k)$.

If $X$ is no longer quasi-projective, then the Weil restriction does not necessarily exist. Anyway choose an affine open covering $(U_i)_{i\in I}$ of $X$, endow the $R$-points of the Weil restrictions $(U_i)_R$ with the order topology, and note that these open sets glue together to define a topology on $X(k)$. This topology does not depend on the choice of the covering by \cite[Lemma 5.6.1]{Sc}, and we call it the $R$-topology on $X(k)$. Again one can consider $X(k)$ as a semialgebraic space in the sense of \cite{DK} or as a real algebraic variety in the sense of \cite{BCR}.

\vskip 2mm
The $R$-topology on $X(k)$ has many good properties, for instance $X(k)$ is semialgebraically connected and of pure dimension twice the dimension of $X$ if $X$ is irreducible \cite{K}. For $k=\C$ and $R=\R$, the $\R$-topology is nothing more than the strong topology.
 The choice of a different real closed field $R$ in $k$ will lead to different topologies on $X(k)$ (for instance the semialgebraic fundamental group does depend on the choice of $R$ \cite{K}). Already with $k=\C$, one can choose a real closed field different from $\R$, even for instance a non-Archimedean $R\subset \C$. We will see however that in our setting, the choice of the real closed field is transparent.

\subsubsection{Basics on the $R$-topology of $k$-varieties}

In this section we fix a real closed field $R$ with algebraic closure $k$. 

Let $X$ be a quasi-projective algebraic variety over $k$. Recall that by Weil restriction \cite{GrFGA,Sc}~:
\begin{enumerate}
\item The variety $X_R$ is nonsingular if $X$ is nonsingular. More precisely, a $k$-point in $X$ is singular if and only if its corresponding $R$-point in $X_R$ is singular.
\item A Zariski open subset $U\subset X$ induces a Zariski open subset $U_R\subset X_R$.
\item A proper morphism $Y\to X$ between quasi-projective algebraic varieties over $k$ induces a proper morphism $Y_R\to X_R$.
\item A finite morphism $Y\to X$ between quasi-projective algebraic varieties over $k$ induces a finite morphism $Y_R\to X_R$.
\end{enumerate}

Let $X$ be an affine algebraic variety over $k$. A regular function on $X$ gives rise  to a polynomial (and thus continuous) mapping $X_R(R)\to R^2$. Indeed the regular function is polynomial, and by Weil restriction a polynomial function to $k$ induced a polynomial mapping to $R^2$ by taking the real and imaginary parts. Finally a polynomial function is continuous with respect to the $R$-topology.

The topological properties of $R$-varieties coming from $k$-varieties are much more moderate than for general $R$-varieties. For instance, if complex irreducible varieties are locally of equal dimension, irreducible algebraic subsets of $\R^n$ may have isolated points. From \cite{BCR}, a real algebraic variety is called central if its subset of nonsingular points is dense with respect to the $R$-topology.

For a subset $A\subset X(k)$, we denote by $\overline{A}^R$ the closure of $A$ with respect to the $R$-topology. We denote by $\Reg(X(k))$ the set of nonsingular points of $X(k)$.
\begin{prop}\label{prop-cent} Let $X$ be an irreducible algebraic variety over $k$. Then $X(k)$ is central :
$$\overline{\Reg(X(k))}^R=X(k).$$ 
\end{prop} 

\begin{proof}
The question being local, it suffices to assume $X$ is affine, and in particular the Weil restriction of $X$ exists.

If $X$ is nonsingular, then so is $X_R$ by Weil restriction, and so $X_R(R)$ is central. 

Otherwise, consider a resolution $\sigma :\tilde X \to X$ of the singularities of $X$ which exists by \cite{Hiro} since $k$ has characteristic zero. Then $\sigma_k$ is surjective since $k$ is algebraically closed, and one can assume that $\sigma_k$ induces a bijection 
$$\tilde U=\sigma_k^{-1}(\Reg(X(k))\to \Reg(X(k))=U.$$
Let $x\in X(k)$, and choose a preimage $\tilde x\in \sigma_k^{-1}(x)$. The centrality of $\tilde X(k)$ implies the existence of a continuous semialgebraic curve $\tilde \gamma :[0,1]\to \tilde X(k)$ with $\tilde \gamma(0)=\tilde x$ and $\tilde \gamma(t)\in \tilde U$ for $t\in (0,1]\subset R$ by the Curve Selection Lemma \cite[Theorem 2.5.5]{BCR}. Its composition $\gamma=\sigma_k \circ \tilde \gamma$ is a continuous semialgebraic curve from $[0,1]$ to $X(k)$ with $\gamma(0)=x$ and $\gamma(t)\in U$ for $t\in (0,1]\subset R$. As a consequence $x$ belongs to the closure with respect to the $R$-topology of $\Reg(X(k))$ in $X(k)$, and so $X(k)$ is central.
\end{proof}

\begin{rem}
In particular, if the irreducible algebraic variety $X$ over $k$ has dimension $d$, then the local semialgebraic dimension of $X(k)$ at any point $x\in X(k)$ is equal to $2d$.
\end{rem}

The following result is not valid in general for algebraic varieties over $R$,  and even for $R=\R$. 
\begin{lem}\label{lem-dense} Let $X$ be an irreducible algebraic variety over $k$. A non-empty Zariski open subset of $X(k)$ is dense with respect to the $R$-topology.
\end{lem} 

\begin{proof} The question being local, it suffices to assume $X$ is affine, and in particular the Weil restriction of $X$ exists.

A Zariski open subset remains Zariski open by Weil restriction. Combined with Proposition \ref{prop-cent}, it suffices to check that a non-empty Zariski open set $U$ in a central irreducible algebraic variety over $R$ is dense with respect to the $R$-topology. This last property is classical ; for instance, the complement is an algebraic subset of strictly smaller dimension, and a semialgebraic triangulation of $X(k)$ adapted to the complement shows that locally, a point in the complement is in the boundary of a semialgebraic simplex in  $U(k)$. 
\end{proof}

Over a general real closed field, the notion of compact sets is advantageously replaced by closed and bounded semialgebraic sets. For instance, the image of a closed and bounded semialgebraic set by a continuous semialgebraic map is again closed and bounded (and semialgebraic) \cite[Theorem 2.5.8]{BCR}.
A semialgebraic map is said to be proper with respect to the $R$-topology if the preimage of a closed and bounded semialgebraic set is closed and bounded.

\begin{lem}\label{lem-surj} Let $\sigma:\tilde X\to X$ be a proper morphism between irreducible varieties over $k$. Then $\sigma_k$ is proper with respect to the $R$-topology. If $\sigma$ is moreover birational, then $\sigma_k$ is surjective.
\end{lem} 

\begin{proof}
The notion of properness is local on the target, so that there is an affine covering of $X$ such that for any open affine subset $U$ in the covering, the restriction $\sigma'=\sigma_{|\sigma^{-1}(U)}$ of $\sigma$ to $\sigma^{-1}(U)$ is proper.
The properness of $\sigma'$ is kept by Weil restriction, so that $\sigma'_k$ is proper with respect to the $R$-topology by \cite[Theorem 9.6]{DK2}. Finally $\sigma_k$ is proper with respect to the $R$-topology since that notion of properness is local on the target too by \cite[Proposition 5.7]{DK}.

If $\sigma$ is birational, there are Zariski open sets $\tilde U\subset \tilde X$ and $U\subset X$ such that $\sigma_{|\tilde U}$ is a bijection onto $U$. Then 
$$U(k)= \sigma_k({\tilde U(k)})\subset \sigma_k(\overline{\tilde U(k)}^R)=\sigma_k(\tilde X(k)),$$
the right hand side equality coming from Lemma \ref{lem-dense}. Finally $\sigma_k(\tilde X(k))$ is closed for the $R$-topology by properness of $\sigma_k$, so that taking the closure with respect to the $R$-topology gives the result by Lemma \ref{lem-dense}.
\end{proof}

\begin{lem}\label{lem-fini} Let $\sigma:\tilde X\to X$ be a finite morphism between algebraic varieties over $k$. Then $\sigma_k$ is closed with respect to the $R$-topology.
\end{lem}

\begin{proof} By \cite[Theorem 4.2]{DK2}, a finite morphism between algebraic varieties over $R$ is closed with respect to the $R$-topology. The result follows since finiteness is local and Weil restriction preserves finite morphisms. 

Alternatively when $X$ and $\tilde X$ are irreducible, a finite morphism is proper, and apply Lemma \ref{lem-surj}.
\end{proof}

\subsection{Subintegrality and homeomorphisms}

We are now in position to generalize the characterization of subintegrality via homeomorphisms, as in \cite[Thm. 3.1]{Be}, over any algebraically closed field of characteristic zero. We also add a radicial property in the equivalences.

\begin{thm}\label{thmFB1} 
Let $\pi:Y\to X$ be a finite morphism between
algebraic varieties over $k$. The following
properties are equivalent:
\begin{enumerate}
\item $\pi$ is subintegral.
\item $\pi_{k}$ is bijective.
\item $\pi_{k}$ is a homeomorphism for the $R$-topology.
\item $\pi_{k}$ is a homeomorphism for the Zariski topology.
\item $\pi$ is a homeomorphism.
\item $\pi$ is radicial.
\end{enumerate}
\end{thm}

\begin{proof}
The equivalence between (4) and (5) is given by the Nullstellensatz. Using the Nullstellensatz, by Proposition \ref{lying-over} and proceeding similarly to Bernard's proof of \cite[Thm. 3.1]{Be} then we get the equivalence between (1), (2) and (5).

It is clear that (3) implies (2). The proof that (2) implies (3) is a direct consequence of Lemma \ref{lem-fini}. Indeed assuming (2), the map $\pi_k$ admits an inverse, and this inverse is continuous with respect to the $R$-topology since $\pi_k$ is a closed map by Lemma \ref{lem-fini}.

The equivalence between (1) and (6) is given by Proposition \ref{sat=semivar}.
\end{proof}

\subsection{Regulous functions on the $k$-rational points}\label{sect-CR}

We fix a real closed field $R$ such that $R[\sqrt{-1}]=k$. 

\subsubsection{Characterization of continuous rational functions}
The choice of $R$ induces a topology, hence a notion of continuity. We define continuous rational functions on an algebraic variety over $k$ as follows.

\begin{defn}\label{def-rat-cont} Let $X$ be an algebraic variety over $k$. Let $U\subset X$ be an open subset of $X$. A continuous rational function on $U(k)$ is a function from $U(k)$ to $k$ which is the continuous extension to $U(k)$ of a rational function on $X$, when $X$ is endowed with the $R$-topology.
\end{defn}

This notion comes initially from real algebraic variety \cite{Ku,KN,FHMM}, and has been studied also in complex algebraic geometry \cite{Be}. In the setting of Definition \ref{def-rat-cont}, it depends a priori on the choice of $R$.
A rational function that is continuous with respect to the $R$-topology can be characterized by the fact that it becomes regular after applying a relevant proper birational map.

\begin{prop} Let $X$ be an algebraic variety over $k$. Let $f:X(k)\to k$ be an everywhere defined function, and assume that $f$ coincides with a regular function on a Zariski open subset of $X(k)$.

 Then, $f$ is continuous with respect to the $R$-topology if and only if there is a proper birational map $\sigma:\tilde X\to X$ such that $f\circ \sigma_k :\tilde X(k)\to k$ is regular. 
\end{prop}

\begin{proof}
Arguing similarly to \cite[Lem. 4.4]{Be}, we may assume $X$ is irreducible.
Assume $f$ to be continuous, and denote by $g$ the rational function on $X$ that coincides with $f$ on a Zariski open subset of $X(k)$. One can resolve the indeterminacy of the rational map $g$ by a sequence of blowings-up along nonsingular centers, giving rise to a proper birational morphism $\sigma:\tilde X\to X$ such that $g\circ \sigma_k :\tilde X(k)\to \PP^1(k)$ is regular. The functions $f\circ \sigma_k$ and $g\circ \sigma_k$ are equal on a Zariski dense subset of $\tilde X(k)$, so they are equal on a subset dense with respect to the $R$-topology by Lemma \ref{lem-dense}. Therefore they coincide on $\tilde X(k)$ by continuity. As a consequence the regular function $g\circ \sigma_k$ takes its values in $k$ rather than in $\PP^1(k)$.

Conversely, let $C\subset k$ be a closed subset with respect to the $R$-topology. The set $(f\circ \sigma_k)^{-1}(C)$ is closed by continuity of $f\circ \sigma_k$, and its image under $\sigma_k$ is equal to $f^{-1}(C)$ by surjectivity of $\sigma_k$ via Lemma \ref{lem-surj}. As a consequence $f^{-1}(C)$ is closed by properness of $\sigma_k$ with respect to the $R$-topology, thanks to Lemma \ref{lem-surj} again.
\end{proof}

Note that the characterization of continuity given above, via a resolution of indeterminacy, does not refer to the choice of $R$. In particular, the continuity with respect to the $R$-topology of a rational function is independent of the choice of the real closed field $R$.

\vskip 2mm
 
Let $U\subset X$ be an open subset of $X$. The continuous rational functions on $U(k)$ are the sections of a presheaf of $k$-algebras on $X(k)$, denoted by $\K_{X(k)}^0$ in the sequel. Since $\K$ is a sheaf, and the presheaf of locally continuous functions on $X(k)$ for the $R$-topology is also a sheaf for the Zariski topology, then $\K_{X(k)}^0$ is a sheaf called the sheaf of continuous rational functions. It makes $(X(k),\K_{X(k)}^0)$ a ringed space.
In case $X$ is affine then we simply denote by $\SR(X(k))$ the global sections of $\K_{X(k)}^0$ on $X(k)$.
A dominant morphism $\pi:Y\to X$ between varieties over $k$ induces an extension $\K_{X(k)}^0\to (\pi_{k})_*\K_{Y(k)}^0$, hence a morphism
$(Y(k),\K_{Y(k)}^0)\to (X(k),\K_{X(k)}^0)$ of ringed spaces.

An important fact is that continuous rational functions on a normal variety are regular. More precisely, next proposition says that continuous rational functions are integral on the regular ones and thus are already regular if the variety is normal.

The set of indeterminacy points of a continuous rational function is related to the normal locus of the ambient variety. 

\begin{prop}\label{prop-lieunorm} Let $X$ be an algebraic variety over $k$.
We have:
\begin{enumerate}
\item $\K_{X(k)}^0\subset (\pi_{k}^\prime)_* \SO_{X'(k)}$ where $\pi':X'\to X$ is the normalization map.
\item If $x$ is a normal closed point of $X(k)$ then $\K_{X(k),x}^0=\SO_{X,x}$.
\end{enumerate}
\end{prop}

\begin{proof}
The properties are local so we may assume $X$ is affine.
The original proof in \cite[Proposition 4.7]{Be} uses only one argument related to the complex setting. It is the density with respect to the strong topology of a Zariski dense open set, that can be replaced by Lemma \ref{lem-dense}. Note that Hartogs Lemma used in the proof is valid over $k$ : if $X$ is normal then the restriction map $\SO(X(k))\to \SO(\Reg(X(k)))$ is surjective since $$\dim(X(k)\setminus \Reg(X(k)))\leq \dim (X(k))-2$$
by \cite[p. 124]{Iit}.
\end{proof}

In the remaining of this section, we show how seminormalization and continuous rational functions are in relation for varieties over $k$. We begin with a description of regular functions on the relative seminormalization in terms of regular functions on the relative normalization.

\begin{prop}
\label{constant}
Let $Y\to X$ be a dominant morphism between algebraic varieties over $k$. Let $U$ be an open subset of $X$. Then
$$\SO_{X^{+}_Y}((\pi^+)^{-1}(U))=\{ f\in \SO_{X'_Y}((\pi')^{-1}(U))\mid \,f {\rm \,is\,constant\, on \,the\,fibers\,of}\,\pi'_{k}\}$$
where $\pi':X'_Y\to X$ (resp. $\pi^+:X^+_Y\to Y$) is the relative normalization (resp. seminormalization) morphism.
\end{prop}

\begin{proof}
By \cite[Cor. 3.7]{Be} (which is written in the case $U$ is affine, but all section 3 there is valid for any open subset $U$), we have 
$$\SO_{X^{+}_Y}((\pi^+)^{-1}(U))=\{f\in \SO_{X'_Y}((\pi')^{-1}(U))\mid \forall x\in U(k), \,\,f\in \SO_{X,x}+\JRad(\SO_{X'_Y,x})\}$$
The radical being an intersection of maximal ideals, we see that the functions in $\SO_{X^{+}_Y}((\pi^+)^{-1}(U))$ correspond to the elements of $ \SO_{X'_Y}((\pi')^{-1}(U))$ constant on the fibers of $\pi'_{k}$.
\end{proof}

We give a characterization of the structural sheaf of the seminormalization of an algebraic variety over $k$ in another one with continuous rational functions generalizing the main result in \cite{Be} to the relative seminormalization and over any algebraically closed field of characteristic zero. In order to state the result, we use the fiber product of two sheaf extensions.

\begin{thm}
\label{thmintclosregulu}
Let $\pi:Y\to X$ be a dominant morphism between algebraic varieties over $k$. 
Then $$(\pi^+_{k})_*\SO_{X^{+}_Y(k)}=\K_{X(k)}^0\times_{(\pi_{k})_*\K_{Y(k)}^0}(\pi_{k})_*\SO_{Y(k)} $$
where $\pi^+:X^+_Y\to X$ is the relative seminormalization morphism.
\end{thm}

\begin{proof}
We may assume $X$ and $Y$ are affine and thus we want to prove that 
$$k[X^{+}_Y]=\SR( X(k))\times_{\SR(Y(k))}k[Y],$$
where the right hand side stands for the fiber product of the rings.
Let $\pi':X_Y^\prime\to X$ be the relative normalization map.

We consider the following diagram $$\begin{array}{ccccc}
	k[X]&\rightarrow  & k[X_Y']&\rightarrow &k[Y]\\
	\downarrow&&\downarrow&& \downarrow \\
	\SR(X(k))&\rightarrow &\SR(X'_Y(k))& \rightarrow&\SR(Y(k))\\
\end{array}$$
where the horizontal maps from the top (resp. the bottom) are given by composition with respectively $\pi'$ and $Y\to X_Y^\prime$ (resp. $\pi_{k}^\prime$ and $Y(k)\to X_Y^\prime(k)$).

We have $k[X^{+}_Y]\subset k[Y]$ by definition. Since $\pi^+$ is subintegral then it follows from Theorem \ref{thmFB1} that $\pi^+_k$ is an homeomorphism with respect to the $R$-topology. Since $\pi^+$ is in addition birational, the composition by $\pi_k^+$ gives an isomorphism between $\K^0(X(k))$ and $\K^0(X_Y^+(k))$. Therefore we get $k[X^{+}_Y]\subset \K^0(X_Y^+(k))=\K^0(X(k))$. In particular $k[X^{+}_Y]\subset\SR( X(k))\times_{\SR(Y(k))}k[Y]$.


Let us prove the converse inclusion. Let $f\in\SR( X(k)))\times_{\SR(Y(k))}k[Y]$. The continuous rational  function $f$ is integral over $k[X]$ by Proposition \ref{prop-lieunorm}, therefore $f\in k[X'_Y]$ since additionally $f\in k[Y]$. As a function on $X'_Y(k)$, the function $f$ is constant on the fibers of $X'_Y(k)\to X(k)$ since $f$ induces a continuous function on $X(k)$. By Proposition \ref{constant}, we obtain then $f\in k[X^+_Y]$. It gives the reverse inclusion 
$\SR( X(k))\times_{\SR(Y(k))}k[Y]\subset k[X^+_Y]$.
\end{proof}

 A continuous rational function on a normal algebraic variety over $k$ is a regular function by Proposition \ref{prop-lieunorm}. The fact that the normalization of $X$ in $Y$ is not necessarily normal imposes to take the fiber product with $\SO_{Y(k)}$ in Theorem \ref{thmintclosregulu}. We state as a theorem the particular case $Y=X'$, the statement becoming much simpler by Proposition \ref{prop-lieunorm}. It says that the ring of continuous rational functions is isomorphic to the ring of regular functions on the seminormalization, it generalizes the main result in \cite{Be} over any algebraically closed field of characteristic zero.

\begin{thm}
\label{caractK0}
Let $X$ be an algebraic variety over $k$. The ringed space $(X^+(k),\SO_{X^{+}(k)})$ is isomorphic to $(X(k),\K_{X(k)}^0)$.
\end{thm}

\begin{rem}
Since $X^+$ doesn't depend of the choice of the real closed field, Theorem \ref{caractK0} gives another way to check that the continuity property of a rational function does not depend on the chosen real closed field.
\end{rem}

\subsubsection{Regulous functions and homeomorphisms}

A regulous function $f$ is a continuous rational function that satisfies the additional property that $f$ remains rational by restriction to any subvariety.
The first author proved that it is always the case \cite[Prop. 4.14]{Be} for complex varieties, contrarily to the real case \cite{KN}.

The following result asserts that continuous rational functions are always regulous.
\begin{cor}\label{cor-reg}
Let $X$ be an algebraic variety over $k$ and let $f\in\K^0(X(k))$. For any Zariski closed subset $V$ of $X$, the restriction $f_{\mid V(k)}$ belongs to $\K^0(V(k))$.
\end{cor}

\begin{proof}
The proof of \cite[Proposition 4.14]{Be} works verbatim using Theorem \ref{caractK0}.
\end{proof}

In the sequel, we also say regulous for continuous rational.

\begin{defn}
Let $X$ and $Y$ be algebraic varieties over $k$. Let $E\subset X(k)$ and $F\subset Y(k)$ be two closed subsets for the $R$-topology.
We say that a map $h:E\to F$ is biregulous if it is a homeomorphism for the $R$-topology which is birational i.e there exist two dense Zariski open subsets $U_1\subset \overline{E}^Z$, $U_2\subset \overline{F}^Z$, a birational map $g:\overline{E}^Z\to\overline{F}^Z$ with $g_{\mid U_1}:U_1\to U_2$ an isomorphism such that $g=h$ by restriction to $U_1\cap E$. 
In this situation, we say that $E$ and $F$ are biregulous or biregulously equivalent.
\end{defn}

In this situation and if in addition $E=Y(k)$, $F=X(k)$ and $h$ is an homeomorphism for the Zariski topology, so that the morphism $\K_{X(k)}^0\to (h)_*\K_{Y(k)}^0$ is well-defined, then this morphism is an isomorphism and thus the ringed spaces $(Y(k),\K_{Y(k)}^0)$ and $(X(k),\K_{X(k)}^0)$ are isomorphic.
The following theorem is an extended version of Theorem \ref{thmFB1} and it explains how subintegral extensions, continuous rational functions and biregulous morphisms are related.
It is a generalization of \cite[Thm. 3.1, Prop. 4.12]{Be} with an additional radiciality property. 

\begin{thm}\label{thmFB}
Let $\pi:Y\to X$ be a finite morphism between
algebraic varieties over $k$. The following
properties are equivalent:
\begin{enumerate}
\item $\pi$ is subintegral.
\item $\pi_{k}$ is bijective.
\item  $\pi_{k}$ is biregulous.
\item The ringed spaces $(Y(k),\K_{Y(k)}^0)$ and $(X(k),\K_{X(k)}^0)$ are isomorphic.
\item $\pi_{k}$ is a homeomorphism for the strong topology.
\item $\pi_{k}$ is a homeomorphism for the Zariski topology.
\item $\pi$ is a homeomorphism.
\item $\pi$ is radicial.
\end{enumerate}
\end{thm}

\begin{proof}
We already have the equivalence of (1), (2), (5), (6), (7) and (8) by Theorem \ref{thmFB1}. 

Assuming $\pi$ to be subintegral, then it follows from Proposition \ref{propCSEPvariety}
that  $X$ and $Y$ have the same  seminormalization. So $\K_{X(k)}^0\to (\pi_k)_*\K_{Y(k)}^0$ is an isomorphism by Theorem \ref{caractK0}. It shows (1) implies (4).

We show (4) implies (1) by contradiction. Assume $\pi_k$ is not bijective. We can suppose $X$ and $Y$ are affine.
We may separate two different points $y,y'$ in the fibre $\pi^{-1}_k(x)$ of some $x\in X(k)$ by a regular function $f$ on $Y(k)$. But such a function is continuous with respect to the $R$-topology, and does not belong to the image of $\K^0(X(k))\to \K^0(Y(k))$ (given by the composition by $\pi_k$) since it is not constant on the fibres of $\pi_k$.

A biregulous map is bijective so (3) implies (2), and conversely by (4) we know that the inverse of $\pi_k$ is continuous rational, so regulous by Corollary \ref{cor-reg}.

\end{proof}

Remark also that a morphism satisfying the conditions of Theorem \ref{thmFB} is (bijective thus) automatically birational.

\vskip 2mm

We prove now that the seminormalization determines a variety up to biregulous equivalence. This result goes in the direction of the problems considered by Koll\'ar \cite{Ko}, \cite{KMOS}, \cite{Ce}.

Before that we need to prove that a biregulous map between algebraic varieties over $k$ is a homeomorphism for the Zariski topology.

Let $X$ be an algebraic variety over $k$ and let $E\subset X(k)$. We denote by $\overline{E}$ the closure of $E$ for the $R$-topology. Recall that a locally closed subset of $X$ or $X(k)$ is the intersection of a Zariski open subset with a Zariski closed subset and that a Zariski constructible subset is a finite union of locally closed subsets.

\begin{prop}
\label{biregZhomeo}
Let $X$ and $Y$ be algebraic varieties over $k$ and let $h:Y(k)\to X(k)$ be a biregulous map. Then $h$ is a homeomorphism for the Zariski topology.
\end{prop}

\begin{proof}
Let $Z$ be an irreducible closed subset of $Y(k)$. By restriction we get a biregulous map $Z\to h(Z)$. We claim that $h(Z)$ is a Zariski closed subset of $X(k)$. Since $h$ is biregulous then $h(Z)$ is closed in $X(k)$ for the $R$-topology. To prove the claim then it is sufficient to assume that $\overline{h(Z)}^Z=X(k)$, $X$ is irreducible and thus we have to prove that $h:Z\to X(k)$ is surjective.
Since a regulous function $f$ on $Z$ is still regulous (and thus rational) by restriction then there exists a stratification of $Z$ in locally closed strata such that the restriction of $f$ to each stratum is regular (see \cite{FHMM} and \cite{Mnew}). It follows that $h(Z)$ is a Zariski constructible set and thus $h(Z)=\cup_{i=1}^n V_i\cap U_i$ such that, for $i=1,\ldots,n$, $V_i$ (resp. $U_i$) is a Zariski closed (resp. open) subset of $X(k)$. Since $\overline{h(Z)}^Z=X(k)$ and $X$ is irreducible then we may assume $V_1=X(k)$.
We have $\overline{U_1}\subset \overline{h(Z)}=h(Z)$ and thus $h:Z\to X(k)$ is surjective by Lemma \ref{lem-dense}.

It shows that $h^{-1}$ is continuous for the Zariski topology. Similarly, $h$ is continuous for the Zariski topology.
\end{proof}

\begin{thm}
\label{QKoC}
Let $X$ and $Y$ be algebraic varieties over $k$. Then, $X(k)$ and $Y(k)$ are biregulously equivalent if and only if $X^+$ and $Y^+$ are isomorphic.
\end{thm}

\begin{proof}
Assume $h:Y(k)\to X(k)$ is a biregulous map. From Proposition \ref{biregZhomeo} $h$ is a homeomorphism for the Zariski topology and thus the ringed spaces $(Y(k),\K_{Y(k)}^0)$ and $(X(k),\K_{X(k)}^0)$ are isomorphic. By Theorem \ref{caractK0} it follows that $X^+$ and $Y^+$ are isomorphic.

Since $X^+(k)\to X(k)$ and $Y^+(k)\to Y(k)$ are biregulous morphisms then we easily get the converse implication.
\end{proof}

\begin{rem}
Note that in the statement of the previous theorem, we do not require that there is a morphism from $X$ to $Y$ nor from $Y$ to $X$. Consider a smooth irreducible curve $Y$ over $k$ without automorphisms (e.g $\PP^1_{k}$ minus sufficiently many points). Let $P_1,P_2$ be two distincts points of $Y$. Let $X_1$ (resp. $X_2$) be the curve obtained from $Y$ by creating a cusp at $P_1$ (resp. $P_2$) i.e associated to the module $2P_1$ (resp. $2P_2$) (see \cite{Se}). We have $Y=X_1^+=X_2^+$ and thus $X_1(k)$ is biregulously equivalent to $X_2(k)$. However, there is no morphism from $X_1$ to $X_2$ (nor from $X_2$ to $X_1$) because otherwise it would lift to a finite and birational morphism $Y\to Y$ sending $P_1$ to $P_2$ (see \cite[Ex. 5.5]{FMQ2}) and since $Y$ is normal then this morphism is an isomorphism, a contradiction.
\end{rem}

\section{Homeomorphisms between algebraic varieties}\label{sect-homeo}

In this section $k$ is an algebraically closed field and $\car (k)=0$. We fix a real closed field $R$ such that $R[\sqrt{-1}]=k$. 

\vskip 1cm

Given a morphism $\pi:Y\to X$ between algebraic varieties over $k$, we are looking for algebraic conditions on $\pi$ which are respectively equivalent to the topological property that $\pi_{k}$ is an homeomorphism for the $R$-topology and $\pi$ is an homeomorphism for the Zariski topology.

In case $\pi$ is finite then we already know the answer. Indeed,  from Theorem \ref{thmFB1} the two topological properties stated above are equivalent to each other, and moreover equivalent to the two algebraic conditions that $\pi$ is subintegral or $\pi$ is radicial. 
\begin{thm}\label{thmFB2} 
Let $\pi:Y\to X$ be a finite morphism between
algebraic varieties over $k$. The following
properties are equivalent:
\begin{enumerate}
\item $\pi_{k}$ is a homeomorphism for the $R$-topology.
\item $\pi$ is a homeomorphism.
\item $\pi$ is subintegral.
\item $\pi$ is radicial
\end{enumerate}
\end{thm}

In the sequel, we compare these four properties when we drop the finiteness hypothesis.

\vskip 2mm

After some generalities on the relation between homeomorphism with respect to Zariski topology, isomorphism and normality, we provide a complete solution to this problem
for $R$-homeomorphisms. As consequences, we give a partial answer to the problem for Zariski homeomorphisms and we provide statements that explain exactly when $R$-homeomorphisms and Zariski homeomorphisms are isomorphisms in the same spirit of Vitulli's result \cite[Thm. 2.4]{V2}.

\subsection{Bijection, birationality, homeomorphism}
We aim to compare the notions of bijection, birational morphism, homeomorphism with respect to Zariski topology at the spectrum level and homeomorphism with respect to Zariski topology at the level of closed points.

\vskip 2mm

To begin with, recall that it follows from the Nullstellensatz that the property for a morphism to be a homeomorphism with respect to the Zariski topology is already decided at the level of closed points. 
\begin{prop}
\label{NS}
Let $\pi:Y\to X$ be a morphism between algebraic varieties over $k$. Then $\pi$ is a homeomorphism if and only if $\pi_k$ is a homeomorphism with respect to the Zariski topology.
\end{prop}

Note that the previous result is still true if $\car(k)>0$.
It is clear that an homeomorphism induces a bijection at the level of $k$-rational points, the converse being false in general as illustrated by Example \ref{exVit2} below. 

Restricting our attention to curves, note that the converse holds true in the case of a morphism between irreducible algebraic curves. The irreducibility of the source space is crucial here, consider for instance the disjoint union of a point with a line minus a point, sent to a line. However even for morphisms between irreducible curves, a birational homeomorphism need not be an isomorphism as illustrated by the normalization of the cuspidal curve with equation $y^2=x^3$.

An important contribution to these questions is the fact that the bijectivity at the level of closed points induces the birationality for irreducible varieties, by Zariski Main Theorem. 

\begin{prop}\label{prop-wk} Let $X$ and $Y$ be irreducible varieties over $k$. Then a morphism from $Y$ to $X$ inducing a bijection at the level of $k$-rational points is quasi-finite and birational. If in addition $X$ is normal, it is an isomorphism.
\end{prop} 

The proof is classical, but we include it for the clarity of the exposition.

\begin{proof}
First note that $Y\to X$ is quasi-finite by \cite[Lem. 20.10]{STPmorph} and the Nullstellensatz.
By Grothendieck's form of Zariski Main Theorem, a quasi-finite morphism $\pi: Y\to  X$ between irreducible algebraic varieties over $k$ factorizes into an open immersion $Y\to Z$ and a finite morphism $Z\to X$. So we identify $Y$ with an open subset of $Z$ and further assume that $Y$ is Zariski dense in $Z$.

Assume $\pi_{k}:Y(k)\to X(k)$ bijective, so that $X$, $Y$ and $Z$ have the same dimension. Recall that the degree of the extension $\K(X)\to \K(Z)=\K(Y)$ is the cardinal of a generic fiber of $Z(k)\to X(k)$ by \cite[Thm. 7]{Sha}. Such a generic fiber is in general in $Y(k)$, otherwise the dimension of $\dim (Z\setminus Y)$ would be greater than or equal to $\dim X$, in contradiction with the density of $Y$.
Thus the finite morphism $Z\to X$ has necessarily degree one, so that $Z\to X$ is birational and thus also $\pi$. 

Assuming in addition $X$ normal implies that $Z$ is isomorphic to $X$. The open immersion is surjective at the level of $k$-rational points and from the Nullstellensatz it follows that it is surjective,
thus an isomorphism. 
\end{proof}

The following example shows that a morphism which gives a bijection at the level of $k$-rational points needs not be an homeomorphism.
\begin{ex} \label{exVit2} 
\begin{enumerate}
\item Consider the varieties of Example \ref{exVitdetail2}, for which we have an open immersion $\psi$ and a finite morphism $\phi$ as follows :
$$\pi : Y\times \Af_k^n \xrightarrow{\psi} X'\times \Af_k^n \xrightarrow{\phi} X\times\Af_k^n.$$ 
Even if $\pi_k$ is still bijective, the morphism $\pi$ is no longer a homeomorphism when $n>0$.

To see it, it suffices to consider the case $n=1$. 
Denote by $O\in (X\times \Af_k^1)(k)$ the origin, and by $P=(0,1,0)$ and $Q=(0,-1,0)$ the two points in the fiber $\phi^{-1}(O)$, where the coordinates are $(x,z,t_1)$ in the notation of Example \ref{exVitdetail2}.
Let $C$ be the curve in $X'\times \Af_k^1$ given by intersection with the plane $x-z+t_1+1=0$ in $\Af_k^2\times\Af_k^1$.  Note that $P\in C$, $Q\notin C$, $C\setminus \{P\}$ is not a closed subset of $X'\times \Af_k^1$ but it is a closed subset of $Y\times \Af_k^1$ since $Y=X'\setminus\{P\}$. If $\pi$ was a homeomorphism, then $\pi_{k}(C(k)\setminus \{P\})$ should be Zariski closed in $(X\times \Af_k^1)(k)$ by Proposition \ref{NS}. However $O$ is in the closure of $\pi_{k}(C(k)\setminus\{P\})$.

\item This example is also interesting to consider relatively to Grothendieck's notion of universal homeomorphism \cite[Defn. 3.8.1]{Gr2}. Recall that a morphism $Y\to X$ is a universal homeomorphism if $Y\times_X Z\to Z$ is a homeomorphism for any morphism $Z\to X$.

The morphism $Y\to X$ of Example \ref{exVitdetail} is an homeomorphism but not a universal homeomorphism. Indeed, let $Z=X\times \Af_k^1$ and consider the base change $Z\to X$ given by the first projection. We have already checked that $Y\times_X Z=Y\times \Af_k^1\to Z=X\times \Af_k^1$ is not closed.
\end{enumerate}
\end{ex}

\subsection{Main results}

We focus now on the four properties appearing in Theorem \ref{thmFB2}, namely given a morphism $\pi:Y\to X$ between algebraic varieties over $k$, we consider the properties :\\
- $\pi_k$ is a homeomorphism for the $R$-topology,\\
- $\pi$ is a homeomorphism,\\
- $\pi$ is subintegral,\\
- $\pi$ is radicial.\\
 By Theorem \ref{thmFB2}, these four properties are equivalent when $\pi$ is finite. Considering the open immersion $\Af_{k}^1\to \PP_{k}^1$, it is radicial but not bijective. Since finite morphisms are surjective by Proposition \ref{lying-over} then if do not assume $\pi$ to be finite then we have to replace the last property above by:\\
 - $\pi$ is radicial and surjective.\\

The equivalence between the four above properties is no longer true without the finiteness hypothesis. In particular, an homeomorphism with respect to the Zariski topology need not be a homeomorphism with respect to the $R$-topology, even for irreducible affine curves. 

\begin{ex} \label{exVit3}
\begin{enumerate}
\item
Consider the morphism $\pi:Y\to X$ from Example \ref{exVitdetail}. We have already shown that $\pi$ is an homeomorphism, $\pi$ is radicial but $\pi$ is not subintegral since it is not an integral morphism.
The morphism $\pi_k$ is not a homeomorphism with respect to the $R$-topology. Indeed, consider a small open ball $B$ of $Y(k)$ containing the point that is sent to the singular point of $X(k)$ by $\pi_k$. Then the image of $Y(k)\setminus B$ is not closed. 
\item Let $n>0$ and consider the morphism $\pi_n: Y\times\Af_k^n\to X\times \Af_k^n$ from Example \ref{exVitdetail2}. We have already proved that $\pi$ is radicial and surjective, $ \pi$ is not subintegral nor an homeomorphism.
Note that $(\pi_n)_k$ is not a homeomorphism with respect to the $R$-topology either, following the same proof as in (1).
\end{enumerate}
\end{ex} 

In view of the foregoing explanation, the only equivalence which remains possible to obtain is between the properties for a morphism to be subintegral and an homeomorphism for the $R$-topology. It will be the subject of our main result.

The following two results measure the rigidity of the $R$-topology.

\begin{prop}\label{prop-fini} Let $\pi:Y\to X$ be a morphism between algebraic varieties over $k$. If $\pi_{k}$ is a homeomorphism with respect to the $R$-topology, then $\pi$ is finite.
\end{prop}

\begin{proof}
It is sufficient to assume that $Y$ and $X$ are irreducible. By Proposition \ref{prop-wk} then $\pi$ is quasi-finite.
By Grothendieck's form of Zariski Main Theorem, $\pi$ factorizes into an open immersion $g:Y\to Z$ and a finite morphism $h:Z\to X$. We consider $Y(k)$ embedded as an open subset of $Z(k)$ for the $R$-topology. We also assume $Y$ to be Zariski dense in $Z$, and thus $Y(k)$ is dense in $Z(k)$ for the $R$-topology by Lemma \ref{lem-dense}.

Since $\pi_{k}$ is bijective, the finite morphism $h$ is birational by Proposition \ref{prop-wk}. Moreover $h_{k}$ is surjective by Proposition \ref{lying-over}. Let us prove that $h_{k}$ is also injective. If not, there exist $y\in Y(k)$ and $z\in Z(k)\setminus Y(k)$ with $h_{k}(y)=h_{k}(z)$. Denote this point by $x\in X(k)$. Let $V_y$ be a closed neighborhood of $y$ in $Y(k)$ for the $R$-topology, and $V_z$ a closed neighborhood of $z$ in $Z(k)$ for the $R$-topology disjoint from $V_y$. Then $V_x=h_{k}(V_y)$ is a closed neighborhood of $x$ in $X(k)$ for the $R$-topology by assumption on $\pi_{k}$. 

By the Curve Selection Lemma \cite[Thm. 2.5.5]{BCR}, there is a continuous semialgebraic curve $\gamma :[0,1)\to Z(k)$ with $\gamma(0)=z$ and $\gamma(0,1)\subset V_z\cap Y(k)$. Then $h_{k}\circ \gamma :[0,1)\to X(k)$ is a continuous semialgebraic curve with $ h_{k} \circ \gamma(0)=x$, so it meets $V_x\setminus\{x\}=h_{k}(V_y\setminus \{y\})$. As a consequence $V_y$ and $V_z$ cannot be disjoint because $\pi_{k}$ is bijective.  

Therefore $h_{k}$ is bijective and thus $g_{k}$ is also bijective. The Nullstellensatz forces $g$ to be a bijective open immersion, thus an isomorphism. As a consequence $\pi$ is finite like $h$.
\end{proof}

\begin{cor}\label{cor-homeo} Let $\pi:Y\to X$ be a morphism between algebraic varieties over $k$. If $\pi_{k}$ is a homeomorphism with respect to the $R$-topology, then $\pi$ is a homeomorphism.
\end{cor}

\begin{proof}
The finiteness follows from Proposition \ref{prop-fini}. Being a bijection on the closed points, it is a homeomorphism by Theorem \ref{thmFB1}.
\end{proof}

Corollary \ref{cor-homeo} admits a converse, for varieties of dimension at least two.

\begin{prop} \label{equivhomeo} Let $\pi:Y\to X$ be a morphism between irreducible algebraic varieties over $k$ of dimension at least two. If $\pi$ is a homeomorphism, then $\pi_{k}$ is a homeomorphism with respect to the $R$-topology.
\end{prop}

\begin{proof}
By \cite[Theorem 2.2]{V2}, the morphism $\pi$ is finite (it is also birational by Proposition \ref{prop-wk}), so we conclude using Theorem \ref{thmFB1}.
\end{proof}

Proposition \ref{equivhomeo} is however not true for curves as illustrated by Example \ref{exVit3} (1).

\vskip 1cm

We are now able to get the main result of the paper. We prove that, for a given morphism between algebraic varieties over $k$, the topological property to be a homeomorphism for the $R$-topology does not depend of the chosen real closed field since it is equivalent to the algebraic property to be subintegral.
\begin{thm}
\label{Euclhomeo}
Let $\pi:Y\to X$ be a morphism between algebraic varieties over $k$. The following properties are equivalent:
\begin{enumerate}
\item $\pi_{k}$ is a homeomorphism with respect to the $R$-topology.
\item $\pi$ is subintegral.
\item $X_Y^+=Y$.
\end{enumerate}
\end{thm}

\begin{proof}
From Theorem \ref{thmFB2} and Proposition \ref{propCSEPvariety} we know that (2) and (3) are equivalent and they imply (1).

Assume $\pi_{k}$ is a homeomorphism with respect to the $R$-topology. Then $\pi$ is finite by Proposition \ref{prop-fini} and thus $\pi$ is subintegral again by Theorem \ref{thmFB2}.
\end{proof}

Note that we cannot replace in Theorem \ref{Euclhomeo} the topological assumption (1) on $\pi_{k}$ by $\pi_{k}$ is bijective or even $\pi$ or $\pi_k$ is a homeomorphism, as illustrated by Examples \ref{exVitdetail} and \ref{exVit3}.

\smallskip

As illustrated by Proposition \ref{prop-wk}, the normality of the target space plays a role to upgrade a bijection into an isomorphism. Next result, which is a direct consequence of Theorem \ref{Euclhomeo}, shows that relative seminormality is the correct notion to associate to a homeomorphism with respect to the $R$-topology in order to obtain an isomorphism. 

\begin{cor}
\label{corEuclhomeo1}
Let $\pi:Y\to X$ be a morphism between algebraic varieties over $k$ such that $\pi_{k}$ is a homeomorphism with respect to the $R$-topology. Then $\pi$ is an isomorphism if and only if $X$ is seminormal in $Y$.
\end{cor}

We obtain an alternative version of \cite[Thm. 2.4]{V2}, where we replace the Zariski topology by the $R$-topology. In particular our statement is valid without any restriction on dimension.

\begin{cor} 
\label{corEuclhomeo2} Let $\pi:Y\to X$ be a morphism between algebraic varieties over $k$. If $\pi_{k}$ is a homeomorphism with respect to the $R$-topology and $X$ is seminormal, then $\pi$ is an isomorphism.
\end{cor}

\begin{proof}
From Proposition \ref{semimprel} then $X$ is seminormal in $Y$ and the result follows from Corollary \ref{corEuclhomeo1}.
\end{proof}

We end the section by some results with a slightly different flavour. Forgetting about the $R$-topology, we compare now homeomorphisms with radicial and sujective morphisms.
So we compare the two properties quoted at the beginning of the section, properties that do not appear in the equivalence of Theorem \ref{Euclhomeo}.
We may also wonder what is the correct assumption to add to an homeomorphism in order to obtain an isomorphism in the spirit of Corollary \ref{corEuclhomeo1}.
We start by proving that an homeomorphism between varieties over $k$ is radicial.
 \begin{prop}
\label{Zhomeo}
Let $\pi:Y\to X$ be a morphism between algebraic varieties over $k$. If $\pi$ is a homeomorphism then $\pi$ is radicial.
\end{prop}

\begin{proof}
We may assume $X$ and $Y$ irreducible, since the irreducible components of $X$ and $Y$ are homeomorphic one-by-one.

Assume first that the dimension of $X$ and $Y$ is at least two. Then $\pi_{k}$ is a homeomorphism with respect to the $R$-topology by Proposition \ref{equivhomeo}, so $\pi$ is finite by Proposition \ref{prop-fini}. By Theorem \ref{Euclhomeo} we get that $\pi$ is subintegral and thus radicial (Proposition \ref{sat=semivar}).

Assume $X$ and $Y$ are curves. Then $X$ and $Y$ are birational by Proposition \ref{prop-wk}. Since moreover $\pi_{k}$ is bijective then $\pi$ is bijective and equiresidual. By definition $\pi$ is radicial.
\end{proof}

Contrary to Theorem \ref{Euclhomeo}, the converse implication in Proposition \ref{Zhomeo} is not valid as Examples \ref{exVit2} (2) and  \ref{exVit3} (2) show. Anyway, we get the analogue of Corollary \ref{corEuclhomeo1} with respect to the Zariski topology. It gives a generalization of \cite[Thm. 2.4]{V2} in any dimension for varieties over $k$.

 \begin{cor}
\label{corZhomeo}
Let $\pi:Y\to X$ be a morphism between algebraic varieties over $k$ such that $\pi$ is a homeomorphism. Then $\pi$ is an isomorphism if and only if $X$ is saturated in $Y$.
\end{cor}

\begin{proof}
Assume $X$ is saturated in $Y$, the converse implication being trivial. By Proposition \ref{Zhomeo} then $\pi$ is radicial. Then $Y\stackrel{Id}{\to} Y\stackrel{\pi}{\to} X$ is a radicial sequence of morphisms.
We conclude that $\pi$ is an isomorphism by Proposition \ref{PU1saturationvar}.
\end{proof}

We get the analogue of Corollary \ref{corEuclhomeo2} with respect to the Zariski topology only in dimension $\geq 2$. It is a kind of reformulation of  \cite[Theorem 2.4]{V2} in characteristic zero.
\begin{cor} 
\label{corZhomeo2} Let $\pi:Y\to X$ be a morphism between irreducible algebraic varieties over $k$ of dimension at least $2$. If $\pi$ is a homeomorphism and $X$ is saturated, then $\pi$ is an isomorphism.
\end{cor}

\begin{proof}
From Propositions \ref{prop-fini}, \ref{equivhomeo} and \ref{Zhomeo} then $\pi$ is radicial and finite. It follows from Proposition \ref{sat=semivar} that $\pi$ is subintegral and $X$ is seminormal.
We conclude $\pi$ is an isomorphism from Corollary \ref{corEuclhomeo1}.
\end{proof}

\begin{rem} Example \ref{exVitdetail} shows that Corollary \ref{corZhomeo2} is false for curves.
\end{rem}

Note that, even though the statements of Proposition \ref{Zhomeo} and Corollaries \ref{corZhomeo}, \ref{corZhomeo2} does not mention the $R$-topology, it plays a crucial role in our proofs.

\section{Homeomorphisms versus isomorphisms in positive characteristic} \label{sect-carpos}

In this section $k$ is an algebraically closed field and $\car (k)=p>0$. 

\vskip 2mm

One can wonder which statements of the previous section can be generalized in positive characteristic. Of course the $R$-topology no longer exists and therefore one can only try to obtain versions of Proposition \ref{Zhomeo} and Corollaries \ref{corZhomeo}, \ref{corZhomeo2}. Note also that Proposition \ref{prop-wk} is no longer valid : indeed, the Frobenious map $\varphi:\Af_k^n\to\Af_k^n$ defined by $\varphi(x_1,\ldots,x_n)=(x_1^p,\ldots,x_n^p)$ is a homeomorphism but is not birational.

We start by giving a version of Proposition \ref{Zhomeo} in positive characteristic.
\begin{prop}
\label{Zhomeocarpos}
Let $\pi:Y\to X$ be a morphism between algebraic varieties over $k$. If $\pi$ is a homeomorphism then $\pi$ is radicial.
\end{prop}

\begin{proof}
Assume $\pi:Y\to X$ is a homeomorphism. We may assume $X$ and $Y$ irreducible, since the irreducible components of $X$ and $Y$ are homeomorphic one-by-one.

Assume first that the dimension of $X$ and $Y$ is at least two. By \cite[Thm. 2.2]{V} then $\pi$ is finite. It is clear that $\pi$ is bijective. Proceeding similarly to Bernard's proof of \cite[Thm. 3.1]{Be} then for any $y\in Y$ we get that the algebraic field extension $k(\pi(y))\to k(y)$ has degree equal to one. It follows that $\pi$ is residually purely inseparable. So $\pi$  is weakly subintegral and thus radicial (Proposition \ref{sat=semivar}).

Assume $X$ and $Y$ are curves. We already know that $\pi$ is bijective. Let $y\in Y(k)$ then since we get an extension $k(\pi(y))\to k(y)=k$ and since $k(\pi(y)$ contains $k$ then $\pi(y)\in X(k)$. Let $x\in X(k)$, by bijectivity of $\pi$ then there exists $y\in Y$ such that $\pi(y)=x$. By Zariski's Lemma then $k(x)=k\to k(y)$ is finite and thus $y\in Y(k)$. We have proved $\pi_k$ is bijective. We adapt now in our situation the beginning of the proof of Proposition \ref{prop-wk}. First note that $Y\to X$ is quasi-finite by \cite[Lem. 20.10]{STPmorph} and the Nullstellensatz.
By Grothendieck's form of Zariski Main Theorem, a quasi-finite morphism $\pi: Y\to  X$ between irreducible algebraic varieties over $k$ factorizes into an open immersion $Y\to Z$ and a finite morphism $Z\to X$. So we identify $Y$ with an open subset of $Z$ and further assume that $Y$ is Zariski dense in $Z$. It follows that $Y\to Z$ is birational.
Since $\pi_{k}:Y(k)\to X(k)$ is bijective, so that $X$, $Y$ and $Z$ have the same dimension. 
Recall that the degree of the extension $\K(X)\to \K(Z)=\K(Y)$ is the cardinal of a generic fiber of $Z(k)\to X(k)$ by \cite[Thm. 7]{Sha}. Such a generic fiber is in general in $Y(k)$, otherwise the dimension of $\dim (Z\setminus Y)$ would be $\geq \dim X$, in contradiction with the density of $Y$.
Since $\pi_{k}:Y(k)\to X(k)$ is bijective then the finite morphism $Z\to X$ has necessarily degree one, so that $\K(X)\to \K(Z)$ is purely inseparable. It follows that $\pi$ is bijective and residually purely inseparable and thus radicial and surjective.
\end{proof}

We end the paper by versions of Corollaries \ref{corZhomeo}, \ref{corZhomeo2} in positive characteristic.

\begin{cor}
\label{corZhomeocarpos}
Let $\pi:Y\to X$ be a morphism between algebraic varieties over $k$ such that $\pi$ is a homeomorphism. Then $\pi$ is an isomorphism if and only if $X$ is saturated in $Y$.
\end{cor}

\begin{proof}
We copy the proof of Corollary \ref{corZhomeo} using Proposition \ref{Zhomeocarpos} instead of Proposition \ref{Zhomeo}.
\end{proof}

\begin{cor} 
\label{corZhomeo2carpos} Let $\pi:Y\to X$ be a morphism between irreducible algebraic varieties over $k$ of dimension $\geq 2$. If $\pi$ is a homeomorphism and $X$ is saturated, then $\pi$ is an isomorphism.
\end{cor}

\begin{proof}
We copy the proof of Corollary \ref{corZhomeo2} using Proposition \ref{Zhomeocarpos},  \cite[Thm. 2.2]{V} and Corollary \ref{corZhomeocarpos} instead respectively of Propositions \ref{Zhomeo}, \ref{prop-fini} and Corollary \ref{corZhomeo}.
\end{proof}

\end{document}